\documentclass{amsart}

\usepackage{graphicx}
\usepackage[english]{babel}
\usepackage[utf8]{inputenc}
\usepackage[T1]{fontenc}
\usepackage{mathtools}
\usepackage{amssymb}
\usepackage[margin=1in]{geometry}
\usepackage{xcolor}

\usepackage{hyperref}

\newcommand{\N}{\mathbb{N}}
\newcommand{\Z}{\mathbb{Z}}

\newcommand{\R}{\mathbb{R}}
\newcommand{\C}{\mathbb{C}}
\newcommand{\ud}{\,\mathrm{d}}

\numberwithin{equation}{section}

\newtheorem{theorem}{Theorem}[section]
\newtheorem{lemma}[theorem]{Lemma}
\theoremstyle{remark}
\newtheorem{remark}[theorem]{Remark}

\begin{document}

\title[Optimal error bound for the biharmonic equation]{Optimal order finite
difference approximation of generalized solutions to the biharmonic equation in
a cube}
\author{Stefan Müller, Florian Schweiger, and Endre Süli}
\address{Hausdorff Center for Mathematics \& Institute for Applied Mathematics,
Universit\"at Bonn, Endenicher Allee 60, 53115 Bonn, Germany. Email:
\texttt{stefan.mueller@hcm.uni-bonn.de}}
\address{Institute for Applied Mathematics,
Universit\"at Bonn, Endenicher Allee 60, 53115 Bonn, Germany.
Email: \texttt{schweiger@iam.uni-bonn.de}}
\address{Mathematical Institute, University of Oxford, Radcliffe Observatory
Quarter,
Woodstock Road, Oxford OX2 6GG, UK. Email: \texttt{endre.suli@maths.ox.ac.uk}}
\date{\today}

\begin{abstract} We prove an optimal order error bound in the discrete $H^2(\Omega)$
norm for finite difference approximations of the first boundary-value problem for
the biharmonic equation in $n$ space dimensions, with $n \in \{2,\dots,7\}$, whose
generalized solution belongs to the Sobolev space $H^s(\Omega) \cap H^2_0(\Omega)$,
for $\frac{1}{2} \max(5,n) < s \leq 4$, where $\Omega = (0,1)^n$. The result
extends the range of the Sobolev index $s$ in the best convergence results
currently available in the literature to the maximal range admitted by the
Sobolev embedding of $H^s(\Omega)$ into $C(\overline\Omega)$ in $n$ space dimensions.
\end{abstract}
\maketitle
\section{Introduction}
The biharmonic equation arises in a number of problems in continuum mechanics,
including linear elasticity and the solution of the Stokes equations modelling
the flow of a viscous incompressible fluid; the biharmonic operator also
features in various nonlinear PDEs of practical relevance such as the Cahn--Hilliard
equation, which describes the process of phase separation in a binary alloy,
and the Ohta--Kawasaki model for the free energy of a diblock copolymer melt.
The convergence analysis of numerical methods for the approximate solution of
the biharmonic equation has been therefore of considerable interest. We shall not
attempt to review the vast literature in this area: the reader may wish to
consult the early papers by Tee \cite{Tee1963}, Bramble \cite{Bramble1966}, Smith
\cite{Smith1968,Smith1970}, and Ehrlich \cite{Ehrlich1971}, for example, for the first
analytical results in this direction. For the numerical analysis of finite
difference approximations of the biharmonic equation in the context of the
approximate solution of the Navier--Stokes equations in planar domains, we point
the reader to the book of Ben-Artzi et al. \cite{Ben-Artzi2013}. In these works the
data and the solution to the boundary-value problems under consideration were
assumed to have (sufficiently) high regularity in spaces of continuously
differentiable functions.

Finite difference schemes for the biharmonic equation with nonsmooth source
terms were considered by Lazarov \cite{Lazarov1981}, Gavrilyuk et al. \cite{Gavrilyuk1983b},
and Ivanovi\'{c} et al. \cite{Ivanovich1986}, for example.
For a detailed survey of the relevant literature we refer the reader to the
monograph of Jovanovi\'{c} and S\"uli \cite{Jovanovic2014}, devoted to the finite
difference approximation of linear partial differential equations with
generalized solutions.

Our objective in this paper is to prove an optimal-order error bound for finite
difference approximations of the first boundary-value problem for the biharmonic
equation in $n$ space dimensions, with $n \in \{2,\dots,7\}$, whose generalized
solution belongs to the Sobolev space $H^s(\Omega) \cap H^2_0(\Omega)$, for
$\frac{1}{2} \max(5,n) < s \leq 4$, where $\Omega = (0,1)^n$. One of our main results, Theorem \ref{t:mainthm},
improves \cite[Theorem 2.69]{Jovanovic2014} (where the theorem is proved for $n=2$ and $s<\frac72$) as
well as the main result in \cite{Gavrilyuk1983} (where the theorem is proved for
$n=2$ under the assumption that the third normal derivative of $u$ vanishes at
the boundary). The restrictions on the range of $s$ in \cite[Theorem 2.69]{Jovanovic2014} and on the third normal derivative of $u$ in \cite{Gavrilyuk1983} arise for the following reason: in order to compare the finite difference approximation with the original problem one needs an extension of the (generalized) solution $u$ from $\Omega$ to $\mathbb{R}^n \setminus \Omega$ that preserves the Sobolev regularity of $u$ and has, ideally, zero discrete boundary values. The assumptions in \cite[Theorem 2.69]{Jovanovic2014} and in \cite{Gavrilyuk1983} permit the use of
the symmetric extension of $u$ across $\partial\Omega$ for that
purpose. In our setting, with $\frac{1}{2} \max(5,n) < s \leq 4$, this is no longer possible, and the main novelty of our work is to use a different, carefully chosen, extension of $u$. This extension will no longer have zero boundary values, but we will show that they can be made small (in an appropriate norm, in terms of positive powers of the discretization parameter $h$), so that we can still close the argument. The strategy of the proof is described in Section 1.2.

The relevance of our results extends beyond the numerical analysis of partial differential equations, to statistical mechanics and probability, particularly the study of the so-called membrane model, a model for a random interface (see the Introduction of \cite{Muller2019} for an overview), which involves a centered Gaussian measure on functions, defined on lattices with uniform spacing, whose covariance matrix is given by the Green's function of the discrete bilaplacian with Dirichlet boundary data. In fact, the analysis pursued here was motivated by recent work by M\"uller and Schweiger \cite{Muller2019}, where estimates for the Green’s function of the discrete bilaplacian on squares and cubes in two and three dimensions were proved. Very recently, Schweiger \cite{Schweiger2019} explored the behavior of the maximum of the solution to the four-dimensional membrane model; for that purpose the estimates from \cite{Muller2019} are not sharp enough, but methods similar to those in the present paper can be employed to relate the Green's function of the discrete bilaplacian with its continuous counterpart and thereby to obtain the required bounds.

The paper is structured as follows. In the remainder of this introductory section we state our main results, outline their proofs, and define the relevant notation that we shall use in the rest of the paper. In Section \ref{s:extension} we recall the definition of an extension operator, which will play a crucial role in our analysis.
In Section \ref{s:estboundaryval} we discuss estimates of boundary values using a discrete counterpart of the fractional-order Sobolev norm $H^{\frac{1}{2}}(\partial\Omega)$. In Section \ref{s:discretetrace} we establish a discrete inverse trace theorem on $\Omega$. In Section \ref{s:summbyparts} we record, for the sake of completeness, some summation-by-parts formulae that we will need, and in Section \ref{s:proofmainthm} we give the proofs of the main theorems. The discussion is completed by two appendices, the first of which concerns various density results that we need and the second contains some remarks on function space interpolation,
which are of relevance in our analysis.

Our results extend to more general fourth-order elliptic elliptic operators with variable coefficients, such as those treated in \cite[Section 2.7]{Jovanovic2014}, with similar proofs. The main difference compared to the analysis here is that in addition to terms appearing in our error bounds one encounters a variety of mixed terms. On can deal with these as in the proof of Theorem 2.68 in \cite{Jovanovic2014}, using the bilinear Bramble--Hilbert lemma. It should also be possible to extend our results to other (higher order) elliptic operators, such as the polyharmonic operator $\Delta^k$ for $k\ge3$, but the study of that question is beyond the scope of this paper.

\subsection{Main results}
We mostly follow the notation of \cite{Jovanovic2014} (see, however, Section \ref{s:Notation} for the precise definitions). Let $n\in\N^+$, $\Omega:=(0,1)^n$, $\Gamma:=\partial\Omega$. For $h\in\R^+$ such that $\frac1h\in\N$, let $\Omega^h:=\Omega\cap(h\Z)^n$, $\Gamma^h:=\Gamma\cap(h\Z)^n$, and
\[\tilde\Omega^h:=[-h,1+h]^n\cap(h\Z)^n\setminus\{-h,1+h\}^n.\]

Consider the elliptic boundary-value problem
	\begin{alignat}{2}\label{eq:bilaplace_cont}
	\begin{aligned}
		\Delta^2u&=f && \text{in }\Omega,\\
		u&=0 && \text{on }\Gamma,\\
		\partial_\nu u&=0&& \text{on }\Gamma,
	\end{aligned}	
	\end{alignat}
where $\partial_\nu$ denotes the derivative in the normal direction ($\nu$ is a unit outward normal vector to $\Gamma$). We approximate the solution of this problem by the finite difference scheme (compare \cite[Section 1.9.4]{Jovanovic2014})
	\begin{alignat}{2}\label{eq:bilaplace_findiff}
        \begin{aligned}
		\Delta_h^2U&=T^{h,2,\ldots,2}f && \text{in }\Omega^h,\\
		U&=0 && \text{on }\Gamma^h,\\
		D^h_{0,\nu} U&=0&& \text{on }\Gamma^h.
        \end{aligned}
	\end{alignat}
Here $U$ is defined on $\tilde\Omega^h$, $D^h_{0,\nu} U(x):=\frac{1}{2h}\big(U(x+h\nu)-U(x-h\nu)\big)$ \footnote{At the singular points (i.e., at the vertices and points on the faces/edges) of $\Gamma^h$ there are up to $n$ possible boundary normal vectors. For \eqref{eq:bilaplace_findiff} we consider all of them. Because $U=0$ on $\Gamma^h$ by assumption, this corresponds to setting $U=0$ at all points of $\tilde\Omega^h\setminus(\Omega^h\cup\Gamma^h)$ that have distance $h$ to a singular point of $\Gamma^h$.}, and $T^{h,2,\ldots,2}f$ is a smoothing operator acting on $f$, defined by convolving $f$ with a B-spline on the scale $h$ (see below for the precise definition).

The finite difference scheme \eqref{eq:bilaplace_findiff} makes sense in any dimension $n$, as the smoothing operator $T^{h,2,\ldots,2}$ maps $f\in H^{s-4}$ into a continuous function whenever $s>\frac52$ (cf. \cite[Theorem 1.69]{Jovanovic2014}).

Our objective is to prove an error bound in the discrete Sobolev norm $\|\cdot\|_{H^2_h(\Omega^h)}$ (which is denoted by $\|\cdot\|_{W^2_2(\Omega^h)}$ in \cite[Section 2.2.4]{Jovanovic2014}).

\begin{theorem}\label{t:mainthm}
Suppose that $\frac{1}{2}\max(5,n)<s\le4$, and let $u\in H^s(\Omega)\cap H^2_0(\Omega)$; then, there
exists a positive constant $C=C(n,s)$, independent of $h$, such that
\begin{equation}\label{eq:mainest}
\|u-U\|_{H^2_h(\Omega^h)}\leq Ch^{s-2}\|u\|_{H^s(\Omega)}.
\end{equation}
\end{theorem}
This improves \cite[Theorem 2.69]{Jovanovic2014} (where the above result was proved for $n=2$ and $\frac{5}{2}< s\le4$ in the more general setting of fourth-order elliptic equations with nonsmooth variable coefficients, but the order of convergence $\mathcal{O}(h^{\min\{s-2,3/2\}}|\log h|^{1-|\rm{sgn}(s-7/2)|})$ established there was optimal only in the case of $\frac{5}{2}<s<\frac{7}{2}$, and is reduced to the suboptimal rate of $\mathcal{O}(h^{\frac{3}{2}})$, instead of the optimal rate of $\mathcal{O}(h^{s-2})$, for $\frac{7}{2}<s\le4$) as well as the main result in \cite{Gavrilyuk1983} (where the theorem was proved for $n=2$ under the additional assumption that the third normal derivative of $u$ vanishes at the boundary).

Our method also yields estimates for other discretizations of the boundary conditions. Consider, for instance, the finite difference scheme
\begin{alignat}{2}\label{eq:bilaplace_findiff2}
	\begin{aligned}
		\Delta_h^2U^*&=T^{h,2,\ldots,2}f && \text{in }\Omega^h,\\
		U^*&=0 && \text{on }\Gamma^h,\\
		D^h_\nu U^*&=0&& \text{on }\Gamma^h.
	\end{aligned}
\end{alignat}
Here again $U^*$ is defined on $\tilde\Omega^h$, and $D^h_\nu U^*(x):=\frac{1}{h}\big(U^*(x+h\nu)-U^*(x)\big)$. The conditions $U^*=0$ and $D^h_\nu U^*=0$ on $\Gamma^h$ are equivalent to $U^*=0$ on $\tilde\Omega^h\setminus\Omega^h$, so that we could equivalently consider the finite difference scheme
\begin{alignat}{2}\label{eq:bilaplace_findiff2'}
	\begin{aligned}
		\Delta_h^2U^*&=T^{h,2,\ldots,2}f && \text{in }\Omega^h,\\
		U^*&=0 && \text{on }\tilde \Omega^h\setminus\Omega^h.
	\end{aligned}
\end{alignat}

For this difference scheme we can show the following error bound.
\begin{theorem}\label{t:mainthm2}
Suppose that $\frac{1}{2}\max(5,n)<s\le3$, and let $u\in H^s(\Omega)\cap H^2_0(\Omega)$; then, there exists a positive constant $C=C(n,s)$, independent of $h$, such that
\begin{equation}\label{eq:mainest2}
\|u-U^*\|_{H^2_h(\Omega^h)}\leq Ch^{s-2}\|u\|_{H^s(\Omega)}.
\end{equation}
\end{theorem}

In Theorems \ref{t:mainthm} and \ref{t:mainthm2} we made the assumption $\frac{1}{2}\max(5,n)<s$. In view of the fact that the problem \eqref{eq:bilaplace_cont} makes sense already for $s>\frac32$, the requirement $\frac{1}{2}\max(5,n)<s$ might seem surprising. The condition $s>\frac n2$ ensures that $u$ is continuous. Otherwise, $\|u-U\|_{H^2_h(\Omega^h)}$ and $\|u-U^*\|_{H^2_h(\Omega^h)}$ would be undefined. The condition $s>\frac52$ implies that $T^{h,2,\ldots,2}f$ is continuous so that its pointwise values are defined and the finite difference schemes \eqref{eq:bilaplace_findiff} and \eqref{eq:bilaplace_findiff2} make sense.
It should be possible to relax the assumption $s>\frac n2$ by replacing $u$ in the expressions $\|u-U\|_{H^2_h(\Omega^h)}$ and $\|u-U^*\|_{H^2_h(\Omega^h)}$ with a suitably mollified version of $u$. Similarly, one can relax the assumption $s>\frac 52$ by replacing $T^{h,2,\ldots,2}$ with a stronger mollification operator. We will not pursue these alternatives here; see however \cite{Schweiger2019} for a version of Theorem \ref{t:mainthm2} with $s<\frac52$.

\subsection{Outline of the proof}
We discuss the proof of Theorem \ref{t:mainthm} only; the proof of Theorem \ref{t:mainthm2} is very similar. We proceed similarly to the proof of \cite[Theorem 2.69]{Jovanovic2014}. In fact, when $s<\frac72$ we could directly use the argument in \cite{Jovanovic2014} with only minor notational changes. Let us review that argument here briefly. We begin by extending $u$ symmetrically across $\Gamma$ to a $H^s$-function $\hat u$ on $(-1,2)^n$ such that $\|\hat u\|_{H^s((-1,2)^n)}\leq C\|u\|_{H^s(\Omega)}$. Here and henceforth $C$ signifies a generic positive constant, which may depend on the Sobolev index $s$
and on the number of space dimensions $n$, but is independent of the discretization parameter $h$.
Let $E:=\hat u-U$. Then, $E$ satisfies %
\begin{alignat*}{2}
		E&=0  \qquad &&\text{on }\Gamma^h,\\
		D^h_{0,\nu} E&=0\qquad&& \text{on }\Gamma^h,
\end{alignat*}
and we calculate (compare \cite[Equation (2.209)]{Jovanovic2014})
\[\Delta_h^2E=\Delta_h^2\hat u-\Delta_h^2U=\Delta_h^2\hat u-T^{h,2,\ldots,2}f=\Delta_h^2\hat u-T^{h,2,\ldots,2}\Delta^2\hat u.\]
Using summations by parts we obtain
\[\|\nabla_h^2E\|_{L^2_h(\Omega^h)}\leq\|\Delta_h^2\hat u-T^{h,2,\ldots,2}\Delta^2\hat u\|_{H^{-2}_h(\Omega^h)},\]
where $\nabla_h^2$ is the discrete Hessian. Now one can use the Bramble--Hilbert lemma (cf. \cite{Jovanovic2014}) to deduce that the right-hand side is bounded by $Ch^{s-2}\|u\|_{H^s(\Omega)}$, which directly implies \eqref{eq:mainest}.

When $s\ge\frac72$ one can no longer extend $u$ symmetrically across the boundary while preserving its Sobolev regularity. This means that we cannot make $D^h_{0,\nu} u$ equal to $0$ on $\Gamma$, and therefore the above argument based on summation by parts no longer works.

Our alternative approach is as follows. Although we cannot ensure that the boundary values of $D^h_{0,\nu} E$ are exactly zero, we will show that they can nevertheless be made small in an appropriate norm. To this end, we will first show (in Section \ref{s:extension}) that we can take a slightly different extension $\tilde u$ with $\|\tilde u\|_{H^s((-1,2)^n)}\le C\|u\|_{H^s(\Omega)}$, $\frac{7}{2} \leq s \leq 4$, such that $\tilde u$ and its derivatives vanish on the hyperplanes supporting the faces of $\Gamma$. This will allow us to estimate the boundary values in an optimal space. In fact, in Section \ref{s:estboundaryval} we prove that
\[[D^h_{0,\nu} \tilde{u}]_{H^{\frac12}_h(\Gamma^h)}\le Ch^{s-2}\|\tilde{u}\|_{H^s(\Omega)}\]
(see that section also for a precise definition of the $H^{\frac12}_h(\Gamma^h)$-seminorm appearing on the left-hand side
of this inequality). Then, in Section \ref{s:discretetrace}, we prove that for any function $g$ on the boundary there is a lattice function $w$ such that
\[\begin{array}{rl}
		w=0 & \text{on }\Gamma^h,\\
		D^h_{0,\nu} w=g& \text{on }\Gamma^h,
	\end{array}\]
and
\[ \|\nabla_h^2 w\|_{L^2_h(\Omega^h)}\le C[g]_{H^{\frac12}_h(\Gamma^h)}.\]
We shall construct $w$ by giving an explicit extension using the Fourier series representation of the boundary values, which we then carefully cut off to comply with the boundary conditions.

Let now $\hat E$ be such a function $w$ corresponding to $g=D^h_{0,\nu} \tilde{u}$.
We can apply the argument above to $E-\hat E$ (which has zero boundary values) and find that
\[\|\nabla_h^2(E-\hat E)\|_{L^2(\Omega^h)}\le C\left(h^{s-2}\|u\|_{H^s(\Omega)}+\|\Delta_h^2\hat E\|_{H^{-2}_h(\Omega^h)}\right).\]
Thus, by observing that
\[\|\Delta_h^2\hat E\|_{H^{-2}_h(\Omega^h)}\le C\|\nabla_h^2\hat E\|_{L^2(\Omega^h)}\le C[D^h_{0,\nu} \tilde{u}]_{H^{\frac12}_h(\Gamma^h)}\le Ch^{s-2}\|u\|_{H^s(\Omega)},\]
we directly deduce \eqref{eq:mainest}. We shall present the details of the argument in Section \ref{s:proofmainthm}. In that section we shall also make several remarks that concern possible modifications and generalizations of the results.

\subsection{Notation and preliminaries}\label{s:Notation}
Our notation is based on that in \cite{Jovanovic2014}, however we made some changes that we will review in the following.

$C$ denotes a constant that may change from line to line and may be dependent on the Sobolev index, $s$, and the number of space dimensions, $n$, but is always independent of $h$. Similarly $C(h)$ denotes a constant that may change from line to line and may depend on $h$ as well.

For $s\ge0$ and $\Xi\subset\R^n$ open with Lipschitz boundary we define the Sobolev space $H^s(\Xi)$ as the space of restrictions of $H^s(\R^n)$-functions to $\Xi$. By $H^s_0(\Xi)$ we denote the closure of the set of all $C_c^\infty(\Xi)$-functions in the $\|\cdot\|_{H^s(\Xi)}$-norm.

Assume that $\Xi:=I_1\times\cdots\times I_n$, where $I_j\subset\R$ are (possibly unbounded) open intervals. This assumption ensures that we have $\mathcal{H}^{n-1}$-almost everywhere on $\partial\Xi$ an axiparallel normal vector. Given a $k\in\N_0$ with $k+\frac12< s$, we denote by $H^s_{(k)}(\Xi)$ the space of all functions $u\in H^s(\Xi)$ such that the traces of $\partial_\nu^iu$ for $0\le i\le k$ vanish on each face of $\partial\Xi$. We extend this definition to $k> s-\frac12$, provided $s\notin\N+\frac12$, by setting $H^s_{(k)}(\Xi)=H^s_{(\lfloor s-1/2\rfloor)}(\Xi)$.

There are several other equivalent definitions of $H^s_{(k)}(\Xi)$.
Let $C_c^\infty(\overline{\Xi})$ denote the space of functions on $\Xi$, which are in $C^\infty(\Xi)$, for which all derivatives admit continuous extensions to $\overline \Xi$, and which are supported in $K\cap \Xi$ for some $K\subset\R^n$ compact. In other words, $C_c^\infty(\overline{\Xi})$ denotes the set of restrictions of $C_c^\infty(\R^n)$-functions to $\Xi$, where the equivalence follows from Whitney's extension theorem \cite{Whitney1934}. Then, $H^s_{(k)}(\Xi)$ is also the closure in the $H^s(\Xi)$-norm of the set of all functions in $C_c^\infty(\overline{\Xi})$ whose derivatives up to order $k$ vanish on $\partial\Xi$
Furthermore, $H^s_{(k)}(\Xi)$ is equal to $H^s(\Xi)\cap H^{k+1}_0(\Xi)$ if $s\ge k+1$, and equal to $H^s_0(\Xi)$ if $s\le k+1$. In particular, the space $H^s(\Omega)\cap H^2_0(\Omega)$ from the main theorems can now be written as $H^s_{(1)}(\Omega)$.

The fact that these definitions are equivalent should not be surprising. Nonetheless we could not locate a reference for this precise equivalence result, and so we present its proof in Appendix \ref{a:density}.

%\medskip

Given a $j\in\N$, we let $\theta_j$ be the standard univariate centered B-spline of degree $j-1$, defined, for example, as the indicator function of the closed interval $\left[-\frac12,\frac12\right]$ convolved with itself $j-1$ times (cf. \cite[Section 1.9.4]{Jovanovic2014}).  Using this, we define the smoothing operator $T^{h,j}_i$ for $1\le i\le n$ as
\[T^{h,j}_if:=\frac{1}{h} f*_i\theta_j\left(\frac{\cdot}{h}\right),\]
where $*_i$ means convolution in the variable $x_i$. This is a well-defined operator on distributions on $\R^n$. Furthermore, we set
\[T^{h,j,\ldots, j}f:=T^{h,j}_1\circ\cdots \circ T^{h,j}_nf.\]
Each $\theta_j$ is in $H^t(\R)$ for any $t<j-\frac12$. Using this, one can verify (cf. \cite[Section 1.9.4]{Jovanovic2014}) that $T^{h,j,\ldots, j}$ is a bounded linear operator from $H^t(\R^n)$ to $C_b(\R^n)$ whenever $t>-j+\frac12$.

%\medskip

Given a unit vector $a\in\R^n$, we define the difference quotients $D^h_av(x):=\frac1h(v(x+ha)-v(x))$, $D^h_{-a}v(x):=\frac1h(v(x)-v(x-ha))$, $D^h_{0,a}v(x):=\frac1{2h}(v(x+ha)-v(x-ha))$. When $a$ is a standard unit vector $e_i$, we write $D^h_i$ instead of $D^h_{e_i}$ (and similarly for $D^h_{-i}$ and $D^h_{0,i}$).

The discrete gradient is the vector $\nabla_hv(x):=(D^h_iv(x))_{i=1}^n$, the discrete Hessian is the tuple $\nabla^2_hv(x):=(D^h_iD^h_{-j}v(x))_{i,j=1}^n$, the discrete Laplacian is $\Delta_hv(x):=\sum_{i=1}^nD^h_iD^h_{-i}v(x)$, and the discrete bilaplacian is $\Delta_h^2:=\Delta_h\circ\Delta_h$.

For any $A\subset(h\Z)^n$ and $v,w\colon A\to\R$ we define $(v,w)_{L^2_h(A)}:=\sum_{x\in A}h^nv(x)w(x)$ and $\|v\|^2_{L^2_h(A)}:=(v,v)_{L^2_h(A)}$. We also define the discrete Sobolev norm $\|v\|_{H^2_h(\Omega^h)}$ of $v\colon\tilde\Omega^h\to\R$ as the sum of the $L^2_h$-norms of $v$, $\nabla_hv$ and $\nabla_h^2v$, wherever they are defined; more precisely,
\[\|v\|^2_{H^2_h}:=\sum_{x\in\tilde\Omega_h}h^nv(x)^2+ \sum_{i=1}^n\sum_{\substack{x\in\tilde\Omega_h:\\x+he_i\in\tilde\Omega_h}} h^n(D^h_iv(x))^2+\sum_{i,j=1}^n\sum_{\substack{x\in\tilde\Omega_h:\\x+he_i,x-he_j,x+he_i-he_j\in\tilde\Omega_h}} h^n(D^h_iD^h_{-j}v(x))^2.\]

Note that we have the crucial property
\[D^h_iD^h_{-i}T^{h,j-2}_if=T^{h,j}_i\partial_i^2f\]
for any $i$ and any $j\ge2$.

\section{Construction of a good extension}\label{s:extension}

At a certain point in our argument it will be necessary to localize the functions concerned in order to deal with the $2^n$ corners of $\Omega=(0,1)^n$ separately. Actually, it is most convenient to do so right from the start. Thus we shall use a partition of unity, which allows us to split $u$ into $2^n$ parts localized near the corners. These parts can all be dealt with in a similar way, so we focus on one of them and assume that $u$ is supported in $\left[0,\frac23\right)^n$.

Recall that $H^s_{(1)}(\Omega)$ denotes the space of functions $u \in H^s(\Omega)$ for which $u$ and the normal derivative vanish on $\partial \Omega$.
\begin{lemma}\label{l:extension} Let $\frac32<s\le 4$, let $u \in H^s_{(1)}(\Omega) $ be supported in $\left[0,\frac23\right)^n$. Then, there exists a function $\tilde u \in H^s_0( (-1,1)^n)$ such that $\|\tilde{u}\|_{H^s((-1,1)^n)}\le C\|u\|_{H^s(\Omega)}$, $\tilde u_{ |\Omega} = u$, and $\tilde u=0$, $\nabla \tilde{u} = 0$ on the $(n-1)$-dimensional hyperplanes $x_i = 0$  for $i \in \{1,\ldots, n\}$ in the sense of traces.
\end{lemma}
Because $\tilde u \in H^s_0( (-1,1)^n)$, we can extend $\tilde u$ outside $(-1,1)^n$ by zero to a function in $H^s(\R^n)$ (that we continue to call $\tilde u$).

The construction of the extension is classical; see, e.g., \cite[Section 11.5]{Lions1972}. Nonetheless we give some details, in particular because a similar construction will be used in Section \ref{s:discretetrace}.
\begin{proof}

There exist $\lambda_{-1},\lambda_{-2} \in \R$ such that
\[
   \lambda_{-1}  +\lambda_{-2} 2^k= (-1)^k \quad \mbox{for $k \in \{2,3\}$}.
\]
We also let $\lambda_1=1$. Then we define the extension $\tilde u$ of $u$ by
\[
\tilde u(x_1,\ldots,x_n)=\sum_{\substack{\varepsilon_1=1 \text{ if }x_1\ge0\\\varepsilon_1\in\{-1,-2\}\text{ if }x_1<0}}\ldots\sum_{\substack{\varepsilon_n=1 \text{ if }x_n\ge0\\\varepsilon_n\in\{-1,-2\}\text{ if }x_n<0}}\lambda_{\varepsilon_1}\cdot\ldots\cdot\lambda_{\varepsilon_n}u(\varepsilon_1x_1,\ldots,\varepsilon_nx_n).\]
For example, for $n=2$ we have
\[\tilde u(x_1,x_2)=
\begin{cases} u(x_1,x_2)&\quad\mbox{for $x_1\ge0,x_2\ge0$},\\
\lambda_{-1}u(-x_1,x_2)+\lambda_{-2}u(-2x_1,x_2)&\quad\mbox{for $x_1<0,x_2\ge0$},\\
\lambda_{-1}u(x_1,-x_2)+\lambda_{-2}u(x_1,-2x_2)&\quad\mbox{for $x_1\ge0,x_2<0$},\\
(\lambda_{-1})^2u(-x_1,-x_2)+\lambda_{-1}\lambda_{-2}u(-x_1,-2x_2)\\
\quad+\lambda_{-1}\lambda_{-2}u(-2x_1,-x_2)+(\lambda_{-2})^2u(-2x_1,-2x_2)&\quad\mbox{for $x_1<0,x_2<0$}.
\end{cases}\]
This extension is constructed by applying an extension operator similar to the one in \cite[Section 2.2]{Lions1972} once across every hyperplane (or in other words by applying a tensorized version of that extension operator).

One easily checks that both $\tilde{u} =0$ and $\nabla \tilde{u} = 0$ on the face $x_i = 0$  for $i \in \{1,\ldots, n\}$. In addition, the support of $\tilde u$ is contained in $\left(-\frac23,\frac23\right)^n\subset(-1,1)^n$.

It remains to show that $\tilde u\in H^s((-1,1)^n)$ and
\begin{equation}\label{eq:esttildeuWs2}
\|\tilde{u}\|_{H^s((-1,1)^n)}\le C\|u\|_{H^s(\Omega)}.\end{equation}
For this we use interpolation.
If $s=4$, and $u\in H^4_{(1)}(\Omega)$ observe that by the construction of $\tilde u$ for $k \in \{0,1,2,3\}$ the traces of $\partial^k_i \tilde u$ from the two sides of $\{x_i=0\}$ agree. This implies that $\tilde u\in H^4((-1,1)^n)$ and $\|\tilde{u}\|_{H^4((-1,1)^n)}\le C\|u\|_{H^4(\Omega)}$. If $s=1$ and $u\in H^1_{(1)}(\Omega)=H^1_0(\Omega)$ we can use the same argument to obtain \eqref{eq:esttildeuWs2} once again.
Now, by Lemma \ref{l:interpolation} from the Appendix, for any $\frac32<s \leq 4$ the interpolation space $\left[H^4_{(1)}(\Omega),H^1_{(1)}(\Omega)\right]_{\frac{1}{3}(4-s)}$ is equal to $H^s_{(1)}(\Omega)$. Thus \eqref{eq:esttildeuWs2} follows by standard function space interpolation theory.
\end{proof}

\section{Estimate of the boundary values}\label{s:estboundaryval}

We want to estimate the boundary values in the fractional discrete Sobolev space $H^{\frac12}_h$. In order to do so, we first define the appropriate (semi-)norms.

Let $S$ be a subset of $\R^n$ that is contained in an axiparallel $(n-1)$-dimensional affine subspace of $\R^n$ such that $S\cap(h\Z)^n\neq\varnothing$, and let $w\colon S\cap (h\Z)^n\to\R$. We then define
\[[w]^2_{H^{\frac12}_h(S\cap (h\Z)^n)}:=\sum_{\substack{x,y\in S\cap (h\Z)^n\\ x\neq y}}\frac{|w(x) - w(y)|^2}{|x-y|^n}h^{2n-2}\]
and
\[\|w\|^2_{H^{\frac12}_h(S\cap (h\Z)^n)}:=\|w\|^2_{L^2_h(S\cap (h\Z)^n)}+[w]^2_{H^{\frac12}_h(S\cap (h\Z)^n)}.\]
For the discrete trace theorem in the following section we will need to use the extension by zero of $D^h_{0,\nu} \tilde u$ and $D^h_\nu\tilde u$.  Therefore we directly estimate the $H^{\frac12}_h$-seminorm of that extension in the following lemma.\footnote{Alternatively one could define a discrete analogue of the $H^{\frac12}_{00}$-norm from \cite[Section 11.5]{Lions1972}; that however leads to unnecessary technicalities in the present context.} At the first glance it might seem problematic that we are extending $D^h_{0,\nu} \tilde u$ by zero, because for $s\ge\frac52$ this extension does not preserve the $H^s$-regularity of $\tilde u$. However it turns out that it is possible to estimate $[D^h_{0,\nu} \tilde u]_{H^{\frac12}_h}$ by expressions that involve several derivatives in the direction $e_n$, but at most one derivative in the directions $e_i$ for $1\le i\le n-1$, so our assumptions on the boundary values are sufficient.

\begin{lemma}\label{l:estboundaryvalues}
Let $s>\frac{1}{2}\max(3,n)$ and let $\tilde{u}$ be as in Lemma \ref{l:extension}. For $i\in\{1,\ldots, n\}$ let $g_{h,i}$ and $g^*_{h,i}$ be the extension by zero of $D^h_{0,i} \tilde u$ and $D^h_{-i}\tilde u$ in the hyperplane $(h\Z)^{i-1}\times\{0\}\times(h\Z)^{n-i}$, respectively, i.e., $g_{h,i}\colon(h\Z)^{i-1}\times\{0\}\times(h\Z)^{n-i}\to\R$ and $g^*_{h,i}\colon(h\Z)^{i-1}\times\{0\}\times(h\Z)^{n-i}\to\R$ satisfy
\begin{align*}
	g_{h,i}(x)&=\begin{cases}D^h_{0,i} \tilde u(x)&\quad\mbox{when ${x\in(0,\infty)^{i-1}\times\{0\}\times[0,\infty)^{n-i}}$},\\0&\quad \text{otherwise},\end{cases}\\
	g^*_{h,i}(0)&=\begin{cases}D^h_{-i} \tilde u(x)&\quad\mbox{when ${x\in(0,\infty)^{i-1}\times\{0\}\times[0,\infty)^{n-i}}$},\\0&\quad\text{otherwise}.\end{cases}
\end{align*}
We have that, if $s\le4$, then
\begin{equation}\label{eq:estghi}
\|g_{h,i}\|_{H^{\frac12}_h((h\Z)^{i-1}\times\{0\}\times(h\Z)^{n-i})}\le Ch^{s-2}\|u\|_{H^s(\Omega)},
\end{equation}
and, if $s\le3$, then
\begin{equation}\label{eq:estg*hi}
\|g^*_{h,i}\|_{H^{\frac12}_h((h\Z)^{i-1}\times\{0\}\times(h\Z)^{n-i})}\le Ch^{s-2}\|u\|_{H^s(\Omega)}.
\end{equation}

\end{lemma}

We can assume that $i=n$, the other cases being analogous. For simplicity we identify $\R^{n-1}$ with the hyperplane $\R^{n-1}\times\{0\}\subset\R^n$, and write $x=(x',x_n)$, with $x':=(x_1,\dots,x_{n-1})$.

Before embarking on the proof of our main result, we state and prove two estimates that we will need.
\begin{lemma}\label{l:estboundaryvalues3}
Let $s>\frac{1}{2}\max(3,n)$ and $v\in H^s(\R^n)$ such that $v=0$ and $\partial_n v=0$ on $\R^{n-1}$ in the sense of trace. Let $h>0$, let $x'\in(h\Z)^{n-1}$, $\hat x'\in\R^{n-1}\times\{0\}$ and suppose that $|x'-\hat x'|_\infty<\frac h2$. Let further $Q_{h/2}(x'):=x'+(-h/2,h/2)^{n-1}$ be the $(n-1)$-dimensional axiparallel cube of edge-length $h$ centered at $x'$. If $s\le4$, we have that
\begin{equation}
	|v(x',h)-v(x',-h)-v(\hat x',h)+v(\hat x',-h)|\le Ch^{s-\frac n2}\|v\|_{H^s(Q_{h/2}(x')\times\R)},\label{eq:estvlocal}
\end{equation}
and if $s\le3$ we have that
\begin{equation}
	|v(x',0)-v(x',-h)-v(\hat x',0)+v(\hat x',-h)|\le Ch^{s-\frac n2}\|v\|_{H^s(Q_{h/2}(x')\times\R)}.\label{eq:estvlocal2}
\end{equation}
\end{lemma}
\begin{proof} We begin with \eqref{eq:estvlocal}. By scaling and translating we can assume that without loss of generality that $h=1$ and $x'=0$.
Because $s>\frac n2$ the left-hand side of \eqref{eq:estvlocal} is bounded by $C\|v\|_{H^s(Q_{1/2}(0)\times (-2,2))}$. Furthermore it vanishes when $v$ is a polynomial of degree at most 3. Indeed the boundary condition ensures that each monomial of degree at most 3 has degree at least 2 in $x_n$ and the left-hand side vanishes for such monomials. So \eqref{eq:estvlocal} follows from the Bramble--Hilbert lemma (applied in $H^s(Q_{1/2}(0) \times (-2,2))$). The estimate \eqref{eq:estvlocal2} can be proved analogously.
\end{proof}

\begin{lemma}\label{l:estboundaryvalues2}
Let $s>\frac32$ and $v\in H^s((0,\infty)^{n-1}\times\R)$. Suppose that for all $i\in\{1,\ldots,n-1\}$ we have $v=0$ on $\{x_i=0\}$ in the sense of trace, and that furthermore we have $\partial_n v=0$ on $\{x_n=0\}$ in the sense of trace. Let $\hat v$ be the extension by zero in the first $n-1$ variables of $v$ to $\R^n$, i.e.,
\[\hat v(x):=\begin{cases}v(x)&x\in(0,\infty)^{n-1}\times\R,\\0&\text{otherwise},\end{cases}\] and let $h>0$. If $s\le4$, then we have that
\begin{equation}\label{eq:estvtrace}
	\|\hat v(\cdot,h)-\hat v(\cdot,-h)\|_{H^{\frac12}(\R^{n-1})}\le Ch^{s-1}\|v\|_{H^s((0,\infty)^{n-1}\times\R)},
\end{equation}
and if $s\le3$, then we have that
\begin{equation}\label{eq:estvtrace2}
	\|\hat v(\cdot,0)-\hat v(\cdot,-h)\|_{H^{\frac12}(\R^{n-1})}\le Ch^{s-1}\|v\|_{H^s((0,\infty)^{n-1}\times\R)}.
\end{equation}
\end{lemma}
\begin{proof}
Let us define the function spaces $G^s((0,\infty)^{n-1}\times\R)$ for $s\in[0,\infty)\setminus\left\{\frac12,\frac32\right\}$ as follows. When $s>\frac32$, $G^s$ is the space that is mentioned in the statement of the lemma, i.e., \[G^s((0,\infty)^{n-1}\times\R):=H^s((0,\infty)^{n-1}\times\R)\cap\{u\colon u=0\text{ on }\partial((0,\infty)^{n-1}\times\R)\}\cap\{u\colon \partial_nu=0\text{ on }(0,\infty)^{n-1}\times\{0\}\}.\]
When $\frac12<s<\frac32$,
\[G^s((0,\infty)^{n-1}\times\R):=H^s((0,\infty)^{n-1}\times\R)\cap\{u\colon u=0\text{ on }\partial((0,\infty)^{n-1}\times\R)\},\]
and if $s<\frac12$,
\[G^s((0,\infty)^{n-1}\times\R):=H^s((0,\infty)^{n-1}\times\R).\]
According to Lemma \ref{l:interpolation2} from the Appendix we have that, for $s\notin\left\{\frac12,\frac32\right\}$,
\begin{align*}
G^s((0,\infty)^{n-1}\times\R)&=\left[G^4((0,\infty)^{n-1}\times\R),G^1((0,\infty)^{n-1}\times\R)\right]_{\frac{4-s}{3}},\\
G^s((0,\infty)^{n-1}\times\R)&=\left[G^3((0,\infty)^{n-1}\times\R),G^1((0,\infty)^{n-1}\times\R)\right]_{\frac{3-s}{2}}.
\end{align*}
Thus it suffices to prove \eqref{eq:estvtrace} for $s=4$ and $s=1$ and \eqref{eq:estvtrace2} for $s=3$ and $s=1$, and then the result follows by interpolation. We prove the former two statements; the proofs of the latter two are completely analogous.

If $s=1$, the condition that $v=0$ on $\{x_i=0\}$ in the sense of trace ensures that $\hat v\in H^1(\R^n)$ and $\|\hat v\|_{H^1(\R^n)}\le \|v\|_{H^1((0,\infty)^{n-1}\times\R)}$. Now we can use standard trace theorems to bound
\begin{align*}
	\|\hat v(\cdot,h)-\hat v(\cdot,-h)\|_{H^{\frac12}(\R^{n-1})}&\le \|\hat v(\cdot,h)\|_{H^{\frac12}(\R^{n-1})}+\|\hat v(\cdot,-h)\|_{H^{\frac12}(\R^{n-1})}\\
	&\le 2\|v\|_{H^1((0,\infty)^{n-1}\times\R)}.
\end{align*}

If $s=4$ the proof is less straightforward. The main difficulty is that $\hat v$ is in general not in $H^4(\R^n)$. Instead we write
\[v(\cdot,h)=v(\cdot,0)+h\partial_nv(\cdot,0)+\frac{h^2}{2}\partial_n^2v(\cdot,0)+\int_0^h\frac{(h-s)^2}{2}\partial_n^3v(\cdot,s)\ud s.
\]
This does not make sense as a pointwise equality, but we can interpret it as an equality in $H^{\frac12}((0,\infty)^{n-1})$, with the integral on the right-hand side being understood as a Bochner integral. Similarly, we have
\[v(\cdot,-h)=v(\cdot,0)-h\partial_nv(\cdot,0)+\frac{h^2}{2}\partial_n^2v(\cdot,0)+\int_{-h}^0\frac{(h+s)^2}{2}\partial_n^3v(\cdot,s)\ud s.
\]
Because we know that $v(\cdot,0)=0$ and $\partial_nv(\cdot,0)=0$ in $H^{\frac12}((0,\infty)^{n-1})$, we deduce from this that
\begin{equation}\label{eq:identityv1}
v(\cdot,h)-v(\cdot,-h)=h^2\int_{-h}^hm\left(\frac{s}{h}\right)\partial_n^3v(\cdot,s)\ud s,\end{equation}
as an identity in $H^{\frac12}((0,\infty)^{n-1})$, where $m(t):=\begin{cases}(1-t)^2 &\mbox{for $t\geq 0$},\\(1+t)^2&\mbox{for $t\leq0$}.\end{cases}$

Let $\hat w$ be the extension by zero in the first $n-1$ variables of $\partial_n^3v$ to $\R^n$, i.e.,
\[\hat w(x):=\begin{cases}\partial_n^3v(x)&\mbox{for $x\in(0,\infty)^{n-1}\times\R$},\\0&\text{otherwise}.\end{cases}\]
Our assumptions on $v$ imply that $\partial_n^3v$ belongs to $H^1_0((0,\infty)^{n-1}\times\R)$. Therefore $\hat w\in H^1(\R^n)$ and $\|\hat w\|_{H^1(\R^n)}=\|\partial_n^3v\|_{H^1_0((0,\infty)^{n-1}\times\R)}$.

Furthermore \eqref{eq:identityv1} continues to hold for the extensions by zero of both sides, so that we also have
\[\hat v(\cdot,h)-\hat v(\cdot,-h)=h^2\int_{-h}^hm\left(\frac{s}{h}\right)\hat w(\cdot,s)\ud s\]
as an identity in $H^{\frac12}(\R^{n-1})$, and hence
\begin{align*}
	\|\hat v(\cdot,h)-\hat v(\cdot,-h)\|_{H^{\frac12}(\R^{n-1})}&\le h^2\int_{-h}^hm\left(\frac{s}{h}\right)\|\hat w(\cdot,s)\|_{H^{\frac12}(\R^{n-1})}\ud s\\
	&\le h^2\int_{-h}^h\|\hat w\|_{H^1(\R^n)}\ud s\\
	&\le 2h^3\|\partial_n^3v\|_{H^1((0,\infty)^{n-1}\times\R)}\\
	&\le 2h^3\|v\|_{H^4((0,\infty)^{n-1}\times\R)},
\end{align*}
which is \eqref{eq:estvtrace}.
\end{proof}

\begin{proof}[Proof of Lemma \ref{l:estboundaryvalues}]
We begin with \eqref{eq:estghi}. As before, we shall assume without loss of generality that $i=n$, and we identify $\R^{n-1}$ with $\R^{n-1}\times\{0\}\subset\R^n$ and write $x=(x',x_n)$.

Note that $D^h_{0,n} \tilde u(x)$ makes sense for any $x\in[0,1)^{n-1}\times\{0\}$, not only for those in $(h\Z)^n$. We denote by $g_n$ the extension by zero of $D^h_{0,n} \tilde u$ in the hyperplane $\R^{n-1}\times\{0\}$, i.e., $g_n\colon\R^{n-1}\times\{0\}\to\R$ satisfies
\[g_n(x)=\begin{cases}D^h_{0,n} \tilde u(x)&\mbox{for $x\in(0,\infty)^{n-1}\times\{0\}$},\\0&\text{otherwise}.\end{cases}\]
Then, $g_{h,n}$ is the restriction of $g_n$ to $(h\Z)^n$, and our goal will be to relate the discrete $H^{1/2}_h$-norm of $g_{h,n}$ and the continuous $H^{1/2}$-norm of $g_n$. We begin by estimating the latter.

Applying Lemma \ref{l:estboundaryvalues2} to $hg_{h,n}$ we obtain
\begin{equation}\| g_n\|_{H^{\frac12}(\R^{n-1})}=\|D^h_{0,n} \tilde u\|_{H^{\frac12}(\R^{n-1})}\le Ch^{s-2}\|\tilde u\|_{H^s(\R^n)}.\label{eq:estg_nW1/2}\end{equation}

Next, let $x'\in(h\Z)^{n-1}$, $\hat x'\in\R^{n-1}$, and suppose that $|x'-\hat x'|_\infty<\frac h2$. Recall that $Q_{h/2}(x')=x'+(-h/2,h/2)^{n-1}$ is the $(n-1)$-dimensional axiparallel cube of edge-length $h$ centered at $x'$. Then, Lemma \ref{l:estboundaryvalues3} implies that
\[|\tilde u(x',h)-\tilde u(x',-h)-\tilde u(\hat x',h)+\tilde u(\hat x',-h)|\le Ch^{s-\frac n2}\|\tilde u\|_{H^s(Q_{h/2}(x')\times \R)}.\label{eq:estD0ulocal}\]
If $x'_i > 0$ for all $i=1, \ldots, n-1$, then
\begin{equation}  \label{eq:gn_versus_D0_utilde}
|g_n(x')-g_n(\hat x')| =\frac1{2h}|\tilde u(x',h)-\tilde u(x',-h)-\tilde u(\hat x',h)+\tilde u(\hat x',-h)|.
\end{equation}
On the other hand, if $x_i \le 0$ for some $i \in \{1, \ldots, n-1\}$, then $\tilde u(x',h) = \tilde u(x', -h) = g_n(x',0)
= 0$ and
\[|g_n(\hat x')|=\begin{cases}\frac{1}{2h}|\tilde u(\hat x',h)-\tilde u(\hat x',-h)|& \mbox{for $x'\in(0,\infty)^{n-1}$},\\0&\text{otherwise}.\end{cases}\]
This, together with \eqref{eq:gn_versus_D0_utilde}, implies that we have in any case
\[|g_n(x')-g_n(\hat x')| \le\frac1{2h}|\tilde u(x',h)-\tilde u(x',-h)-\tilde u(\hat x',h)+\tilde u(\hat x',-h)|.\]
Thus we get that
\begin{equation}\label{eq:estgnlocal}
|g_n(x')-g_n(\hat x')|
\le Ch^{s-1-\frac n2}[\tilde u]_{H^s(Q_{h/2}(x')\times \R)}
		\le Ch^{s-1-\frac n2}\|\tilde u\|_{H^s(Q_{h/2}(x')\times \R)}.
\end{equation}

Now let $x', y' \in (h\Z)^{n-1}$, $\hat x'\in Q_{h/2}(x')$ and $\hat y'\in Q_{h/2}(y')$. We then have that
\[ |g_n(x') - g_n(y')| \le |g_n(\hat x') - g(\hat y')| + |g(x') - g(\hat x')| + |g(y') -g(\hat y')|.\]
This implies that
$ |g_n(x') - g_n(y')|^2  \le 3\big(  |g_n(\hat x') - g_n(\hat y')|^2 + |g_n(x') - g_n(\hat x')|^2 + |g_n(y') -g_n(\hat y')|^2 \big)$,
and, using \eqref{eq:estgnlocal}, we deduce that
\begin{equation}\label{eq:est_gn_averag}
	|g_n(x') - g_n(y')|^2 \le 3 |g_n(\hat x') - g_n(\hat y')|^2 + Ch^{2s-2-n}\|\tilde u\|^2_{H^s(Q_h(x')\times (-2h,2h))}+Ch^{2s-2-n}\|\tilde u\|^2_{H^s(Q_h(y')\times (-2h,2h))}.
\end{equation}
Thus, taking the average of \eqref{eq:est_gn_averag} over all $\hat x' \in Q_{h/2}(x')$ and $\hat y' \in Q_{h/2}(y')$, we obtain
\begin{align*}
 |g_n(x') - g_n(y')|^2&\le 3h^{2-2n} \int_{Q_{h/2}(x')}\int_{Q_{h/2}(y')}
|g_n(\hat x') - g_n(\hat y')|^2  \, \mathrm{d}\hat x' \, \mathrm{d}\hat y' \\
&\quad+Ch^{2s-2-n} \left(\|\tilde u\|^2_{H^s(Q_h(x')\times \R)}+\|\tilde u\|^2_{H^s(Q_h(y')\times \R)}\right).
\end{align*}
Observe that for $|x'-y'| \ge h$ we have
\[  |\hat x' - \hat y'| = |x'-y'-(x'-\hat x')+(y'-\hat y')| \leq |x'-y'| + |x'-\hat x'| + |y'-\hat y'| \leq |x'-y'| + h \leq 2|x'-y'|. \]
Using this, we deduce that
\begin{align}
[g_{h,n}]^2_{H^{\frac12}_h((h\Z)^{n-1})}&=\sum_{\substack{x',y'  \in (h\Z)^{n-1}\\x' \ne y'}}  \frac{|g_n(x') - g_n(y')|^2}{|x'-y'|^n}h^{2n-2}\label{eq:est_ghn}\\
&\le  3\cdot2^n \int_{\R^{n-1}} \int_{\R^{n-1}}   \frac{|g_n(\hat x') - g_n(\hat y')|^2}{|\hat x'-\hat y'|^n} \, \mathrm{d}\hat x' \, \mathrm{d}\hat y' \nonumber\\
&\quad +   Ch^{n+2s-4} \sum_{\substack{x',y'  \in (h\Z)^{n-1}\\x' \ne y'}} \frac{1}{|y'-x'|^n}  \|\tilde u\|^2_{H^s(Q_h(x')\times \R)} \nonumber\\
&\quad +   Ch^{n+2s-4} \sum_{\substack{x',y'  \in (h\Z)^{n-1}\\x' \ne y'}}  \frac{1}{|x'-y'|^n}  \|\tilde u\|^2_{H^s(Q_h(y')\times \R)}.\nonumber
\end{align}
The first term on the right-hand side is a constant times $[g_n]^2_{H^{\frac12}(\R^{n-1})}$. To estimate the second term, notice that
\[ \sum_{y' \in (h\Z)^{n-1}, y' \ne x'}\frac{1}{|y'-x'|^n} =  \frac1{h^{n-1}} \sum_{y' \in (h\Z)^{n-1}, y' \ne x'}   \frac{1}{|y'-x'|^n} h^{n-1}
\le \frac{C}{h^{n-1}} \int_{|y'-x'| \ge h}  \frac{1}{|y'-x'|^n} \, \mathrm{d}y' \le \frac{C}{h^n}\]
and
\[\sum_{x' \in (h\Z)^{n-1}}\|\tilde u\|^2_{H^s(Q_h(x')\times \R)}\le\|\tilde u\|^2_{H^s(\R^n)},\]
by superadditivity of the fractional Sobolev norm.

Together with the analogous estimate for the third term and \eqref{eq:estg_nW1/2} we arrive at
\begin{align}
	[g_{h,n} ]^2_{H^{\frac12}_h((h\Z)^{n-1})} &\le C [g_n]^2_{H^{\frac12}(\R^{n-1})} + Ch^{2s-4} \|\tilde u\|^2_{H^s(\R^n)}\nonumber\\
	&\le Ch^{2s-4}\|\tilde u\|^2_{H^s(\R^n)}\nonumber\\
	&\le Ch^{2s-4}\|u\|^2_{H^s(\Omega)}.\label{eq:est_ghn2}
\end{align}
It remains to estimate $\|g_{h,n}\|_{L^2_h((h\Z)^{n-1})}$. A simple way to do so is to observe that we have a Poincar\'{e}-type inequality. Indeed, $g_{h,n}$ is supported in $\left[0,\frac23\right]^{n-1}\cap(h\Z)^{n-1}$ and therefore
\begin{align*}
[g_{h,n} ]^2_{H^{\frac12}_h((h\Z)^{n-1})}&=\sum_{\substack{x',y'  \in (h\Z)^{n-1}\\x' \ne y'}}  \frac{|g_n(x') - g_n(y')|^2}{|x'-y'|^n}h^{2n-2}\\
&\ge\sum_{x' \in [0,1)^{n-1}\cap (h\Z)^{n-1}}\sum_{y' \in [-2,-1)^{n-1}\cap (h\Z)^{n-1}}  \frac{|g_n(x') - g_n(y')|^2}{|x'-y'|^n}h^{2n-2}\\
&\ge\sum_{x' \in [0,1)^{n-1}\cap (h\Z)^{n-1}}\sum_{y' \in [-2,-1)^{n-1}\cap (h\Z)^{n-1}}  \frac{|g_n(x')|^2}{(3\sqrt{n})^n}h^{2n-2}\\
&\ge \frac{1}{(3\sqrt{n})^n}\sum_{x' \in [0,1)^{n-1}\cap (h\Z)^{n-1}}h^{n-1}|g_{h,n}(x')|^2\\
&\ge\frac{1}{(3\sqrt{n})^n}\|g_{h,n}\|^2_{L^2((h\Z)^{n-1})}.
\end{align*}
Combining this with \eqref{eq:est_ghn2} we obtain \eqref{eq:estghi}. The proof of \eqref{eq:estg*hi} is similar, with the only difference that we use
\eqref{eq:estvlocal2} and \eqref{eq:estvtrace2} instead of \eqref{eq:estvlocal} and \eqref{eq:estvtrace}.
\end{proof}

\section{A discrete inverse trace theorem on the cube}\label{s:discretetrace}
We want to construct a function $\hat E$ such that $\hat E$ and $E$ agree on $\Gamma^h$ and such that the $H^2_h$-norm of $\hat E$ is small (and similarly for $\hat E^*$). We now state the two lemmas that we will need to establish Theorems \ref{t:mainthm} and \ref{t:mainthm2}, respectively.

\begin{lemma}\label{l:discretetrace}
Let $\frac{1}{2}\max(3,n)<s\le 4$ and let $\tilde{u}$ be as in Lemma \ref{l:extension}. Then, there is a function $\hat E$ on $\Omega^h$ such that
\begin{alignat*}{2}
		\hat E&=0 \qquad &&\text{on }\Gamma^h,\\
		D^h_{0,\nu} \hat E&=D^h_{0,\nu} \tilde u \qquad &&\text{on }\Gamma^h,
\end{alignat*}
and such that $\|\nabla_h^2\hat E\|_{L^2(\Omega^h)}\le Ch^{s-2}\|u\|_{H^s(\Omega)}$.
\end{lemma}

\begin{lemma}\label{l:discretetrace2}
Let $\frac{1}{2}\max(3,n)<s\le 3$ and let $\tilde{u}$ be as in Lemma \ref{l:extension}. Then, there is a function $\hat E^*$ on $\Omega^h$ such that
\begin{alignat*}{2}
		\hat E^*&=0 \qquad &&\text{on }\Gamma^h,\\
		D^h_\nu \hat E^*&=D^h_\nu \tilde u \qquad &&\text{on }\Gamma^h,
\end{alignat*}
and such that $\|\nabla_h^2\hat E^*\|_{L^2(\Omega^h)}\le Ch^{s-2}\|u\|_{H^s(\Omega)}$.
\end{lemma}

\begin{proof}[Proof of Lemma \ref{l:discretetrace}]
Because we have localized $\tilde u$, $D^h_{0,\nu} \tilde u$ has nonzero boundary values only on the faces
$\Gamma^h\cap\{x_i=0\}$. We can deal with the faces separately. In fact we will construct functions
$\hat E_i$ for $i\in\{1,\ldots, n\}$ such that $\hat E_i=0$ on $\Gamma_h$, $D^h_{0,\nu} \hat E_i=D^h_{0,\nu} \tilde u$ on $x_i=0$ while $D^h_{0,\nu} \hat E_i=0$ on $\Gamma_h\setminus \{x_i=0\}$, satisfying the estimate $\|\nabla_h^2\hat E_i\|_{L^2(\Omega^h)}\le Ch^{s-2}\|u\|_{H^s(\Omega)}$. Then we can choose $\hat E=\sum_i\hat E_i$, which will have the desired properties. As the $\hat E_i$ can be constructed analogously, we shall focus on $\hat E_n$ only.

Recall the function $g_{h,n}$, the extension by zero of $D^h_{0,n} \hat u$. Thanks to our assumption, $g_{h,n}$ is supported in $\left[0,\frac23\right]^{n-1}\cap (h\Z)^{n-1}$. We can extend this function periodically with period 2 and represent it by its discrete Fourier series

\[ g_{h,n}(x') = \sum_{k'\in\left\{-\frac1h+1,\ldots,\frac1h\right\}^{n-1}}\gamma_{k'}{\rm e}^{i\pi(k'\cdot x')},\]
where
\[\gamma_{k'}=\left(\frac h2\right)^{n-1}\sum_{\xi'\in[-1,1)^{n-1}\cap (h\Z)^{n-1}}g_{h,n}(\xi')\,{\rm e}^{-i\pi k'\cdot\xi'}\]
and $k':=(k'_1,\ldots,k'_{n-1})\in\Z^{n-1}$.

It is easy to verify that %
\begin{equation}\label{eq:fourierW1/2}
	\sum_{k'\in\left\{-\frac1h+1,\ldots,\frac1h\right\}^{n-1}}  (1+|k'|)\gamma_{k'}^2\le C\|g_{h,n}\|^2_{H^{\frac12}_h((h\Z)^{n-1})}.
\end{equation}
Indeed, the Fourier norm is controlled by the $H^{\frac12}_h$-norm on the torus $([-1,1]^{n-1}\cap (h\Z)^{n-1})/_\sim$ (compare, e.g., \cite[Section 2.3]{Hackbusch1981}) and the latter is bounded by the $H^{\frac12}_h$-seminorm on $(h\Z)^{n-1}$ because the support of $g_{h,n}$ is bounded away from $\partial[-1,1]^{n-1}$.

Define
\[ a(x',x_n) := \sum_{k'\in\left\{-\frac1h+1,\ldots,\frac1h\right\}^{n-1}}  \frac{\gamma_{k'}}{\cosh (|k'|h)}x_n\,{\rm e}^{-|k'|x_n}\,{\rm e}^{i\pi k'\cdot x'}.\]
It is then easy to check that $a(x',0)=0$ and $D^h_{0,n} a(x',0)=g_{h,n}(x')$ for $x'\in(-1,1)^{n-1}\cap (h\Z)^{n-1}$. Furthermore, the $H^2_h$-norm of $a$ is controlled. Indeed, we have that
\[ \nabla_h^2 a(x',x_n) =   \sum_{k'\in\left\{-\frac1h+1,\ldots,\frac1h\right\}^{n-1}} \sigma(k',h,x_n)\gamma_{k'}\,{\rm e}^{i\pi k'\cdot x'},\]
where the coefficients $\sigma(k',h,x_n)$ satisfy $|\sigma(k',h,x_n)|\le C|k'|(|k'|x_n+1){\rm e}^{-|k'|x_n}$ (this is seen by observing that the term with
the `worst' behavior is $D^h_{-n}D^h_n\left(x_n{\rm e}^{-|k'|x_n}{\rm e}^{i\pi k'\cdot x'}\right)\approx (|k'|^2x_n-2|k'|){\rm e}^{-|k'|x_n}{\rm e}^{i\pi k'\cdot x'}$).
Therefore, using orthogonality in $x'$ we get, for $x_n\ge0$,
\begin{align*}
	\sum_{x'\in[-1,1)^{n-1}\cap(h\Z)^{n-1}} h^{n-1}|\nabla_h^2 a(x_1,x_2)|^2  &=2^{n-1}\sum_{k'\in\left\{-\frac1h+1,\ldots,\frac1h-1\right\}^{n-1}}|\sigma(k',h,x_n)|^2|\gamma_{k'}|^2\\
	&\le C\sum_{k'\in\left\{-\frac1h+1,\ldots,\frac1h-1\right\}^{n-1}}|k'|^2(|k'|^2x_n^2+1)\,{\rm e}^{-2|k'|x_n}|\gamma_{k'}|^2,
\end{align*}
and hence
\[\sum_{x\in [-1,1)^{n-1}\times[0,2]\cap(h\Z)^n} h^n |\nabla_h^2 a(x',x_n)|^2 \le Ch\sum_{x_n\in[0,2]\times h\Z}\sum_{k'\in\left\{-\frac1h+1,\ldots,\frac1h-1\right\}^{n-1}}|k'|^2(|k'|^2x_n^2+1)\,{\rm e}^{-2|k'|x_n}|\gamma_{k'}|^2.\]
Next, we use the estimate
\[\sum_{x_n\in[0,2]\cap h\Z}hx_n^\alpha\,{\rm e}^{-2|k'|x_n}\le C_\alpha\int_0^\infty \xi^\alpha {\rm e}^{-2|k'|\xi} \, \mathrm{d}\xi
= C_\alpha\frac{1}{|k'|^{1+\alpha}} \int_0^\infty \theta^2 \,{\rm e}^{-2\theta} \, \mathrm{d}\theta\le \frac{C_\alpha}{|k'|^{1+\alpha}}\]
for $\alpha=2$ and $\alpha=0$ to deduce that
\[\sum_{x\in [-1,1)^{n-1}\times[0,2]\cap(h\Z)^n} h^n |\nabla_h^2 a(x',x_n)|^2\le C\sum_{k'\in\left\{-\frac1h+1,\ldots,\frac1h-1\right\}^{n-1}}|k'||\gamma_{k'}|^2,\]
and thus, taking into account \eqref{eq:fourierW1/2} and \eqref{eq:estghi},
\begin{equation}\|\nabla_h^2a\|_{L^2_h((-1,1)^{n-1}\times[0,2]\cap(h\Z)^n)}\le C[g_{h,n}]_{H^{\frac{1}{2}}_h((h\Z)^{n-1})}\le Ch^{s-2}\|u\|_{H^s(\Omega)}.\label{eq:estnabla2a}
\end{equation}
Similarly, we estimate
\begin{align}
\|\nabla_ha\|_{L^2_h((-1,1)^{n-1}\times[0,2]\cap(h\Z)^n)}&\le Ch^{s-2}\|u\|_{H^s(\Omega)},\label{eq:estnablaa}\\
\|a\|_{L^2_h((-1,1)^{n-1}\times[0,2]\cap(h\Z)^n)}&\le Ch^{s-2}\|u\|_{H^s(\Omega)}\label{eq:esta}
\end{align}
(note that for these estimates we actually need control of $\|g_{h,n}\|_{H^{\frac{1}{2}}_h((h\Z)^{n-1})}$, not just of $[g_{h,n}]_{H^{\frac{1}{2}}_h((h\Z)^{n-1})}$).

Let $\eta\in\C_c^\infty(\R)$ be such that $\eta=1$ in $\left[-\frac34,\frac34\right]$, $\eta=0$ in $\R\setminus[-1,1]$, and let \[\tilde a(x):=\eta(x_1)\cdots\eta(x_n)a(x).\]
Because $a=0$ on $\{x_i=0\}$ for all $i$, we have that $\tilde a=0$ on $\Gamma^h$. Furthermore, $D^h_{0,n}  a=0$ except possibly in $\left[-\frac23,\frac23\right]^{n-1}\times\{0\}$, and the product $\eta(x_1)\cdots\eta(x_n)$ is equal to the constant 1 in a neighborhood of that set. Therefore, $D^h_{0,n} \tilde a=D^h_{0,n}  a=g_{h,n}$ on $\{x_n=0\}$.

Using the estimates \eqref{eq:estnabla2a}, \eqref{eq:estnablaa}, \eqref{eq:esta} and the discrete product rule, we also obtain
\begin{align}
&\|\nabla_h^2\tilde a\|_{L^2_h(\Omega^h)}\nonumber\\
&\quad\le C\left(\|\nabla_h^2a\|_{L^2_h((-1,1)^{n-1}\times[0,2]\cap(h\Z)^n)}+\|\nabla_ha\|_{L^2_h((-1,1)^{n-1}\times[0,2]\cap(h\Z)^n)}+\|a\|_{L^2_h((-1,1)^{n-1}\times[0,2]\cap(h\Z)^n)}\right)\nonumber\\
&\quad\le Ch^{s-2}\|u\|_{H^s(\Omega)}.\label{eq:estnabla2b}
\end{align}

Unfortunately, $\tilde a$ does not yet have the correct boundary values at $\{x_i=0\}$ for $1\le i\le n-1$. To rectify this we use a discrete projection from $H^2$ to $H^2_0$. First we define the corresponding continuous projection. It is defined in a similar way as the extension we used in the proof of Lemma \ref{l:extension}, namely by tensorizing the restriction operator from \cite[Section 11.5]{Lions1972}. Thus we choose $\lambda_{-1},\lambda_{-2} \in \R$ such that
\[
   \lambda_{-1}  +\lambda_{-2} 2^k= (-1)^{k+1} \quad \mbox{for $k \in \{0,1\}$}.
\]
(i.e.,  $\lambda_{-1}=-3$, $\lambda_{-2}=2$);  we let $\lambda_1=1$ and define a restriction operator $R\colon H^2(\R^{n-1})\to H^2_0((0,\infty)^{n-1})$ by
\[Ru(x):=\sum_{\varepsilon_1\in\{1,-1,-2\}}\ldots\sum_{\varepsilon_n\in\{1,-1,-2\}}\lambda_{\varepsilon_1}\cdot\ldots\cdot\lambda_{\varepsilon_{n-1}}u(\varepsilon_1x_1,\ldots,\varepsilon_{n-1}x_{n-1}).\]
It is easy to check that we indeed have $Ru\in H^2_0((0,\infty)^{n-1})$ and $\|Ru\|_{H^2((0,\infty)^{n-1})}\le C\|u\|_{H^2(\R^{n-1})}$. If we extend $Ru$ by zero to $\R^{n-1}$ we can also consider $R$ as an operator mapping $H^2(\R^{n-1})$ to itself. Note that if $x'\in(h\Z)^{n-1}$, then $Ru(x')$ depends only on $u|_{(h\Z)^{n-1}}$. Thus we can define $R_h\colon H^2_h((h\Z)^{n-1})\to H^2_h((h\Z)^{n-1})$ by
\[R_hu_h(x'):=\begin{cases}Ru(x')&\mbox{for $x'\in[0,\infty)^{n-1}$},\\
Ru(x'+2he_i)&\mbox{for $x'\in [0,\infty)^{i-1}\times\{-h\}\times[0,\infty)^{n-1-i}$},\\
0&\text{otherwise}.\end{cases}\]
We claim that
\begin{align}
\|R_hu_h\|_{L^2_h([0,\infty)^{n-1}\cap(h\Z)^{n-1})}&\le C\|u_h\|_{L^2_h((h\Z)^{n-1})},\label{eq:estR}\\
\|\nabla_hR_hu_h\|_{L^2_h([0,\infty)^{n-1}\cap(h\Z)^{n-1})}&\le C\|\nabla_hu_h\|_{L^2_h((h\Z)^{n-1})},\label{eq:estnablaR}\\
\|\nabla_h^2R_hu_h\|_{L^2_h([0,\infty)^{n-1}\cap(h\Z)^{n-1})}&\le C\|\nabla_h^2u_h\|_{L^2_h((h\Z)^{n-1})}.\label{eq:estnabla2R}
\end{align}
Indeed, these estimates follow from the discrete chain rule. The only exception are the terms $D^h_iD^h_{-i}R_hu_h(x')$ in \eqref{eq:estnabla2R}, which are not, a priori, controlled on $\{x_i=0\}$. However an explicit calculation shows that for such $x'$ one has
\begin{align*}
D^h_iD^h_{-i}R_hu_h(x')&=2\frac{R_hu_h(x'+he_i)}{h^2}\\
&=2\frac{u_h(x'+he_i)-3u_h(x'-he_i)+2u_h(x'-2he_i)}{h^2}\\
&=2\frac{u_h(x'+he_i)-2u_h(x')+u_h(x'-he_i)}{h^2}+4\frac{u_h(x')-2u_h(x'-he_i)+u_h(x'-2he_i)}{h^2}\\
&=2D^h_iD^h_{-i}u_h(x')+4D^h_iD^h_{-i}u_h(x'-he_i),
\end{align*}
so that these terms, which are `crossing the boundary', are still controlled.\footnote{It is of course no coincidence that we have such an identity. In fact, $\nabla^2(Ru)$ is bounded in the $L^2$ norm thanks to the construction of $R$, and one can therefore also expect $R_h$ to be well-behaved at the boundary.}

We now apply $R_h$ along every slice $(h\Z)^{n-1}\times\{x_n\}$, i.e.,  we set
\[b(x):=R_h\tilde a(\cdot,x_n)(x').\]
Then by construction of $R_h$ we have $b=0$ and $D^h_{0,i}b(x)=0$ on $\{x_i=0\}$. Furthermore, $b$ is supported in $\left[-h,\frac34\right]^n$ and we have  $b=0$ on $\{x_n=0\}$. We know that $D^h_{0,n}\tilde a=g_{h,n}$ on $\{x_n=0\}$. In addition, $R_h g_{h,n}=g_{h,n}$ on $[0,\infty)^{n-1}\times\{0\}$, and so $D^h_{0,n}b=g_{h,n}$ on $[0,\infty)^{n-1}\cap(h\Z)^{n-1}$ follows from the fact that $R_h$ and $D^h_{0,n}$ commute.

We next estimate $\|\nabla_h^2b\|_{L^2(\Omega^h)}=\|\nabla_h^2R_h\tilde a\|_{L^2(\Omega^h)}$. If $i,j\le n-1$ then \eqref{eq:estnabla2R} implies that
\[\|D^h_iD^h_{-j}R_h\tilde a\|_{L^2(\Omega^h)}\le C\nabla_h^2\tilde a\|_{L^2_h((h\Z)^{n-1})}.\]
When taking derivatives in the direction $e_n$ we use \eqref{eq:estnablaR} and the fact that $D^h_{\pm n}$ and $R_h$ commute, to obtain (for $i<n$) that
\begin{align*}
\|D^h_iD^h_{-n}R_h\tilde a\|_{L^2(\Omega^h)}&=\|D^h_iR_hD^h_{-n}\tilde a\|_{L^2(\Omega^h)}\\
&\le C\|D^h_iD^h_{-n}\tilde a\|_{L^2(\Omega^h)},
\end{align*}
and similarly, using \eqref{eq:estR},
\[\|D^h_nD^h_{-n}R_h\tilde a\|_{L^2(\Omega^h)}\le C\|D^h_nD^h_{-n}\tilde a\|_{L^2(\Omega^h)}.\]
If we combine the last three estimates and use \eqref{eq:estnabla2b} we deduce that
\[\|\nabla_h^2b\|_{L^2(\Omega^h)}=\|\nabla_h^2\tilde a\|_{L^2(\Omega^h)}\le Ch^{s-2}\|u\|_{H^s(\Omega)}.\]
Thus we can set $b=\hat E_n$, and have shown that $\hat E_n$ has all of the desired properties.
\end{proof}
\begin{proof}[Proof of Lemma \ref{l:discretetrace2}]
The proof is quite similar to the proof of Lemma \ref{l:discretetrace}. Let us outline the differences. In Step 2 we use a different extension operator, namely
\[ a^*(x',x_n) := \sum_{k'\in\left\{-\frac1h+1,\ldots,\frac1h-1\right\}^{n-1}}  \frac{\gamma_{k'}}{{\rm e}^{|k'|h}}x_n\,{\rm e}^{-|k'|x_n}\,{\rm e}^{i\pi k'\cdot x'},\]
so that $a^*(x',0)=0$ and $D^h_{-n}a(x',0)=g^*_{h,n}(x')$ for $x'\in(-1,1)^{n-1}\cap (h\Z)^{n-1}$.
Using Lemma \ref{l:estboundaryvalues} we then again obtain
\[\|\nabla_h^2a^*\|_{L^2_h((-1,1)^{n-1}\times[0,2]\cap(h\Z)^n)}\le C[g_{h,n}]_{H^{\frac{1}{2}}_h(\R^{n-1})}\le Ch^{s-2}\|u\|_{H^s(\Omega)}\]
for $s\le3$. The localization step remains unchanged. To correct the boundary values we use
\[R_h^*u_h(x'):=\begin{cases}Ru(x')&\mbox{for $x'\in[0,\infty)^{n-1}$},\\
0&\text{otherwise},\end{cases}\]
instead of $R_h$. By using this projection operator we can then proceed as before.
\end{proof}

\section{Summation-by-parts formulae and Poincar\'{e} inequalities}\label{s:summbyparts}
For the sake of completeness we record some summation-by-parts formulae that we will use in the following. These formulae are adapted to the two boundary conditions that we encounter in \eqref{eq:bilaplace_findiff} and \eqref{eq:bilaplace_findiff2}.
Zero boundary conditions are easier to deal with, so we begin with those.

\begin{lemma}\label{l:summbyparts1}
Let $v,\varphi\colon\tilde\Omega^h\to\R$, and assume that $\varphi  = D^h_\nu \varphi = 0 $ on $\Gamma^h$.

	We have that
\begin{equation}\label{eq:summbyparts1_1}
\sum_{z\in\Omega^h\cup\Gamma^h}h^n\Delta_h^2v(z)\varphi(z)= \sum_{i,j=1}^n\sum_{z\in\Omega^h\cup\Gamma^h}h^nD^h_iD^h_{-j}v(z)D^h_iD^h_{-j}\varphi(z).
\end{equation}
So, if we define the scalar product $(f,g)_{L^2_{h,*}(\Omega^h\cup\Gamma^h)}$ on functions $f,g\colon\tilde\Omega^h\to\R^{n\times n}$ by
\[(f,g)_{L^2_{h,*}(\Omega^h\cup\Gamma^h)}:=\sum_{z\in\Omega^h}  \sum_{i,j=1}^n h^nf_{i,j}(z)g_{i,j}(z),\]
we have
\begin{equation}\label{eq:summbyparts1_2}
(\Delta_h^2v,\varphi)_{L^2_h(\Omega^h\cup\Gamma^h)}=(\nabla_h^2v,\nabla_h^2\varphi)_{L^2_{h,*}(\Omega^h\cup\Gamma^h)}.
\end{equation}
Furthermore, we have, for any $i\in\{1,\ldots,n\}$, that
\begin{equation}
(D^h_iD^h_{-i}v,\varphi)_{L^2_h(\Omega^h\cup\Gamma^h)}=(v,D^h_iD^h_{-i}\varphi)_{L^2_h(\Omega^h\cup\Gamma^h)}.\label{eq:summbyparts1_3}
\end{equation}

\end{lemma}
\begin{proof}

Observe that we have the summation-by-parts identity
\begin{equation}\label{eq:summbyparts1_4}
\sum_{z\in (h\Z)^n}D^h_{\pm i}f(z)g(z)=\sum_{z\in (h\Z)^n}f(z)D^h_{\mp i}g(z)
\end{equation}
for $f,g\colon(h\Z)^n\to\R$ such that at least one of $f,g$ has compact support, and $i\in\{1,\ldots,n\}$ (this follows from the one-dimensional case, where it can be
easily checked). This immediately implies \eqref{eq:summbyparts1_3}.

Next, observe that none of the terms in \eqref{eq:summbyparts1_1} depends on values of $v$ or $\varphi$ outside of $\tilde\Omega^h$. Thus we can extend $v$ and $\varphi$ by 0 to all of $(h\Z)^n$ and prove equivalently that
\[\sum_{z\in(h\Z)^n}h^n\Delta_h^2v(z)\varphi(z)= \sum_{z\in(h\Z)^n}h^n\Delta_hv(z)\Delta_h\varphi(z)= \sum_{i,j=1}^n\sum_{z\in(h\Z)^n}h^nD^h_iD^h_{-j}v(z)D^h_iD^h_{-j}\varphi(z).\]
This follows from repeated application of \eqref{eq:summbyparts1_4}.
\end{proof}

For the case of the boundary conditions in \eqref{eq:bilaplace_findiff}, the situation is slightly more involved.
We define, for $i,j \in \{1,\ldots, n\}$ with $i \neq j$, the set
\[ \Gamma^h_{ij } := \{ z \in \Gamma^h : z + h A_{ij} \subset [0,1]^n \}, \]
where $A_{ij}$ is the discrete square
\[A_{ij} := \{ 0, e_i, -e_j, e_i-e_j\},\]
and note that
\[
 z \in \Gamma^h \setminus \Gamma^h_{ij} \quad \Longrightarrow  \quad (z + h A_{ij}) \cap \Omega^h = \varnothing \quad \hbox{if $i \neq j$}.
\]
\begin{lemma}\label{l:summbyparts2}
Let $v,\varphi\colon\tilde\Omega^h\to\R$, and assume that $\varphi=D^h_{0,\nu}\varphi=0$ on $\Gamma^h$.
We then have that
\begin{align}\label{eq:summbyparts2_1}
&\sum_{z\in\Omega^h\cup\Gamma^h}h^n\Delta_h^2v(z)\varphi(z)= \sum_{i,j=1}^n\sum_{z\in\Omega^h}h^nD^h_iD^h_{-j}v(z)D^h_iD^h_{-j}\tilde\varphi(z)\nonumber\\
&\qquad+\frac12\sum_{i=1}^n\sum_{z\in\Omega^h}h^nD^h_iD^h_{-i}v(z)D^h_iD^h_{-i}\tilde\varphi(z)+\sum_{i,j=1}^n\sum_{z\in\Gamma^h_{ij}}h^nD^h_iD^h_{-j}v(z)D^h_iD^h_{-j}\tilde\varphi(z).
\end{align}
So, if we define the scalar product $(f,g)_{L^2_{h,\sim}(\Omega^h\cup\Gamma^h)}$ on functions $f,g\colon\tilde\Omega^h\to\R^{n\times n}$ by
\[(f,g)_{L^2_{h,\sim}(\Omega^h\cup\Gamma^h)}:=\sum_{i,j=1}^n\sum_{z\in\Omega^h}h^nf_{ij}(z)g_{ij}(z)+\frac12\sum_{i=1}^n\sum_{z\in\Omega^h}h^nh^nf_{ii}(z)g_{ii}(z)+\sum_{\substack{i,j=1\\i\neq j}}^n\sum_{z\in\Gamma^h_{ij}}h^nf_{ij}(z)g_{ij}(z),
\]
then we have that
\begin{equation}\label{eq:summbyparts2_2}
(\Delta_h^2v,\varphi)_{L^2_h(\Omega^h\cup\Gamma^h)}=(\nabla_h^2v,\nabla_h^2\varphi)_{L^2_{h,\sim}(\Omega^h\cup\Gamma^h)}.
\end{equation}
In addition, if we also define for
$f,g\colon\tilde\Omega^h\to\R$ the scalar product
\begin{align*}
&(f,g)_{L^2_{h,\sim}(\Omega^h\cup\Gamma^h)}:=\sum_{z\in\Omega^h}h^nf_{ij}(z)g_{ij}(z)+\frac12\sum_{z\in\Omega^h}h^nf_{ii}(z)g_{ii}(z),
\end{align*}
then we have, for any $i\in\{1,\ldots,n\}$, that
\begin{equation}
(D^h_iD^h_{-i}v,\varphi)_{L^2_h(\Omega^h\cup\Gamma^h)}=(D^h_iD^h_{-i}v,\varphi)_{L^2_{h,\sim}(\Omega^h\cup\Gamma^h)}=(v,D^h_iD^h_{-i}\varphi)_{L^2_{h,\sim}(\Omega^h\cup\Gamma^h)}\label{eq:summbyparts2_3}.
\end{equation}

\end{lemma}
\begin{proof}
Define $\tilde\varphi\colon(h\Z)^n\to\R$ as $\tilde{\varphi}(z):=\begin{cases} \varphi(z)&\mbox{for $z\in\Omega^h$}\\0&\text{otherwise}\end{cases}$. Then we can apply Lemma \ref{l:summbyparts1} to $v,\tilde\varphi$ and obtain
\[
\sum_{z\in\Omega^h\cup\Gamma^h}h^n\Delta_h^2v(z)\tilde\varphi(z)= \sum_{z\in\Omega^h\cup\Gamma^h}h^n\Delta_hv(z)\Delta_h\tilde\varphi(z)= \sum_{i,j=1}^n\sum_{z\in\Omega^h\cup\Gamma^h}h^nD^h_iD^h_{-j}v(z)D^h_iD^h_{-j}\tilde\varphi(z).\]
We trivially have
\[\sum_{z\in\Omega^h\cup\Gamma^h}h^n\Delta_h^2v(z)\tilde\varphi(z) =\sum_{z\in\Omega^h\cup\Gamma^h}h^n\Delta_h^2v(z)\varphi(z).\]
Furthermore, $D^h_iD^h_{-j}\tilde\varphi(z)$ is equal to $D^h_iD^h_{-j}\varphi(z)$ if $z\in\Omega^h$. If $z\in\Gamma^h$ we have $D^h_iD^h_{-i}\tilde\varphi(z)=\frac12D^h_iD^h_{-i}\varphi(z)$ and $D^h_iD^h_{-j}\tilde\varphi(z)=\begin{cases} D^h_iD^h_{-j}\varphi(z) &\mbox{for $z\in\Gamma^h_{ij}$}\\0 &\text{otherwise}\end{cases}$. Therefore,
\begin{align*}
&\sum_{i,j=1}^n\sum_{z\in\Omega^h\cup\Gamma^h}h^nD^h_iD^h_{-j}v(z)D^h_iD^h_{-j}\tilde\varphi(z)=\sum_{i,j=1}^n\sum_{z\in\Omega^h}h^nD^h_iD^h_{-j}v(z)D^h_iD^h_{-j}\tilde\varphi(z)\\
&\quad+\frac12\sum_{i=1}^n\sum_{z\in\Omega^h}h^nD^h_iD^h_{-i}v(z)D^h_iD^h_{-i}\tilde\varphi(z)+\sum_{i,j=1}^n\sum_{z\in\Gamma^h_{ij}}h^nD^h_iD^h_{-j}v(z)D^h_iD^h_{-j}\tilde\varphi(z).
\end{align*}
By combining the last four displayed equalities we deduce \eqref{eq:summbyparts2_1}. With a similar argument we can obtain \eqref{eq:summbyparts2_3} from \eqref{eq:summbyparts1_3}.
\end{proof}

Next, we state Poincar\'{e}-type inequalities for the two sets of boundary conditions considered.
\begin{lemma}\label{l:Poincare1}
Let $v\colon\tilde\Omega^h\to\R$, and suppose that $\varphi  = D^h_\nu \varphi = 0 $ on $\Gamma^h$. Then,
\begin{equation}\label{eq:Poincare1_1}
\|v\|_{H^2_h(\Omega^h)}\le C\|\nabla_h^2v\|_{L^2_{h,*}(\Omega^h)}.\end{equation}
\end{lemma}
\begin{lemma}\label{l:Poincare2}
Let $v\colon\tilde\Omega^h\to\R$, and suppose that $\varphi  = D^h_{0,\nu} \varphi = 0 $ on $\Gamma^h$. Then,
\begin{equation}\label{eq:Poincare2_1}
\|v\|_{H^2_h(\Omega^h)}\le C\|\nabla_h^2v\|_{L^2_{h,\sim}(\Omega^h)}.\end{equation}
\end{lemma}
\begin{proof}[Proof of Lemma \ref{l:Poincare1}]
We can extend $v$ by 0 to $(h\Z)^n$ without changing the statement of the lemma. Now observe that for $f\colon(h\Z)^n \rightarrow \R$ with support
contained in a cube of side-length $L$, and $i\in\{1,\ldots,n\}$, we have the Poincar\'{e} inequality
\[\|f\|_{L^2_h((h\Z)^n)}\le CL\|D^h_{\pm i}f\|_{L^2_h((h\Z)^n)}.\]
Indeed this follows from the one-dimensional case, which can be proved by a straightforward summation by parts.
If we apply this inequality to $v$ and $\nabla v$, we easily deduce \eqref{eq:Poincare1_1}.
\end{proof}

\begin{proof}[Proof of Lemma \ref{l:Poincare2}]
Let $\tilde{v}(z):=\begin{cases} v(z)&\mbox{for $z\in\Omega^h$}\\0&\text{otherwise}\end{cases}$. Then, $\tilde v$ satisfies the assumptions of Lemma \ref{l:Poincare1}, so that
\[\|\tilde v\|_{H^2_h(\Omega^h)}\le C\|\nabla_h^2\tilde v\|_{L^2_{h,*}(\Omega^h)}.\]
Furthermore it is easy to check that \[\|\nabla_h^2\tilde v\|_{L^2_{h,*}(\Omega^h)}\le\|\nabla_h^2v\|_{L^2_{h,\sim}(\Omega^h)}\]
and
\[\|v\|_{H^2_h(\Omega^h)}\le 2\|\tilde v\|_{H^2_h(\Omega^h)},\]
and hence we directly deduce \eqref{eq:Poincare2_1}.
\end{proof}

\section{Proofs of the main theorems}\label{s:proofmainthm}
We have already sketched the proof of Theorem \ref{t:mainthm} in the Introduction. We now provide additional details.
\begin{proof}[Proof of Theorem \ref{t:mainthm}]

As was mentioned at the start of Section \ref{s:extension}, we can assume that $u$ is supported in $\left[0,\frac23\right)^2$.
Let $E\colon\tilde\Omega^h\to\R$ be defined by $E:=u-U$. Then,
\begin{alignat*}{2}
		E&=0 \qquad&& \text{on }\Gamma^h,\\
		D^h_{0,\nu} E&=D^h_{0,\nu} \tilde{u}\qquad&& \text{on }\Gamma^h.
\end{alignat*}
Let $\hat E$ be the function from Lemma \ref{l:discretetrace}. Then,
\[\begin{array}{rl}
		E-\hat E=0 & \text{on }\Gamma^h,\\
		D^h_{0,\nu} (E-\hat E)=0& \text{on }\Gamma^h.
\end{array}\]
Therefore, using the results from Section \ref{s:summbyparts} we deduce that
\begin{align}
\|\nabla_h^2(E-\hat E)\|^2_{L^2_{h,\sim}(\Omega^h\cup\Gamma^h)}&=(\Delta_h^2(E-\hat E),E-\hat E)_{L^2_h(\Omega^h\cup\Gamma^h)}\nonumber\\
&=(\Delta_h^2E,E-\hat E)_{L^2_h(\Omega^h\cup\Gamma^h)}-(\nabla_h^2\hat E,\nabla_h^2(E-\hat E))_{L^2_{h,\sim}(\Omega^h\cup\Gamma^h)}.\label{eq:identitye}
\end{align}
We can rewrite $\Delta_h^2E$ as follows
\begin{align*}
	 \Delta_h^2E&=\Delta_h^2\tilde{u}-\Delta_h^2U=\Delta_h^2\tilde{u}-T^{2,\ldots,2}f=\Delta_h^2\tilde{u}-T^{h,2,\ldots,2}\Delta^2\tilde{u}\\
	&=\sum_{i=1}^nD^h_iD^h_{-i}\Delta_h\tilde u-T^{h,2,\ldots,2}\partial_i^2\Delta\tilde{u}\\
	&=\sum_{i=1}^nD^h_iD^h_{-i}\Delta_h\tilde u-D^h_iD^h_{-i}T^{h,2}_1\ldots T^{h,2}_{i-1} T^{h,2}_{i+1}\ldots  T^{h,2}_n \Delta\tilde{u}\\
	&=\sum_{i=1}^nD^h_iD^h_{-i}\varphi_i,
\end{align*}
where we have abbreviated
\[	\varphi_i:=\Delta_h\tilde u-T^{h,2}_1\ldots T^{h,2}_{i-1} T^{h,2}_{i+1}\ldots  T^{h,2}_n \Delta\tilde{u}.\]
If we insert this into \eqref{eq:identitye} and use the summation-by-parts formula \eqref{eq:summbyparts2_3} we arrive at
\begin{align*}
\|\nabla_h^2(E-\hat E)\|^2_{L^2_{h,\sim}(\Omega^h\cup\Gamma^h)}&=\sum_{i=1}^n(\varphi_i,D^h_iD^h_{-i}(E-\hat E))_{L^2_h(\Omega^h\cup\Gamma^h)}-(\nabla_h^2\hat E,\nabla_h^2(E-\hat E))_{L^2_{h,\sim}(\Omega^h\cup\Gamma^h)}\\
&\le \left(\sum_{i=1}^n\|\varphi_i\|_{L^2_h(\Omega^h\cup\Gamma^h)}+\|\nabla_h^2\hat E\|_{L^2_h(\Omega^h\cup\Gamma^h)}\right)\|\nabla_h^2(E-\hat E)\|_{L^2_{h,\sim}(\Omega^h\cup\Gamma^h)},
\end{align*}
and thus
\begin{equation}
	\|\nabla_h^2E\|_{L^2_{h,\sim}(\Omega^h\cup\Gamma^h)}\le \|\nabla_h^2\hat E\|_{L^2_h(\Omega^h\cup\Gamma^h)}+\sum_{i=1}^n\|\varphi_i\|_{L^2_h(\Omega^h\cup\Gamma^h)}.\label{eq:decomp_E}
\end{equation}

The first term on the right-hand side here is bounded by $Ch^{s-2}\|u\|_{H^s(\Omega)}$ by construction of $\hat E$. The summands of the sum can be bounded using the Bramble--Hilbert lemma as in the proof of \cite[Theorem 2.68]{Jovanovic2014}. Let us sketch the argument for completeness:

Recall that
\[\varphi_i(x)=\Delta_h\tilde u(x)-T^{h,2}_1\ldots T^{h,2}_{i-1} T^{h,2}_{i+1}\ldots  T^{h,2}_n \Delta\tilde{u}(x).\]
Because $s>\frac n2$, \[|\Delta_h\tilde u(x)|\le C(h)\|\tilde u\|_{L^\infty(x+(-h,h)^n)}\le C(h)\|\tilde u\|_{H^s(x+(-h,h)^n)}.\]
In addition $s>\frac 52$ implies according to \cite[Theorem 1.67]{Jovanovic2014} that
\[|T^{h,2}_1\ldots T^{h,2}_{i-1} T^{h,2}_{i+1}\ldots  T^{h,2}_n \Delta\tilde{u}(x)|\le C(h)\|\tilde u\|_{H^s(x+(-h,h)^n)}.\]
Thus $\varphi_i(x)$ is a bounded linear functional of $\tilde u\in H^s(x+(-h,h)^n)$. This functional vanishes when $\tilde u|_{x+(-h,h)^n}$ is a polynomial of degree at most 3. Indeed, then $\Delta\tilde{u}(y)$ is equal to some affine function $a(y)$, and $\Delta_h\tilde u(x)=a(x)$. On the other hand, the smoothing operators $T^{h,2}_j$ map affine functions to themselves, so that $\varphi_i(x)=0$.

To summarize $\varphi_i(x)$ is a bounded linear functional of $\tilde u\in H^s(x+(-h,h)^n)$ that vanishes on polynomials of degree at most 3.
Hence by the Bramble--Hilbert lemma it is bounded by $C(h)[\tilde u]_{H^s(x+(-h,h)^n)}$ for the range of $s$ as in the statement of the theorem.
Using a scaling argument to determine the correct prefactor of $h$, we obtain
\begin{equation}
	\|\varphi_i\|_{L^2_{h,\sim}(\Omega^h\cup\Gamma^h)}\le Ch^{s-2}[\tilde u]_{H^s(\R^n)} \le Ch^{s-2}\|u\|_{H^s(\Omega)}\label{eq:est_phii}
\end{equation}
for those $s$.

Now we substitute \eqref{eq:est_phii} into \eqref{eq:decomp_E} and obtain the bound
\[\|\nabla_h^2E\|_{L^2_{h,\sim}(\Omega^h\cup\Gamma^h)}\le Ch^{s-2}\|u\|_{H^s(\Omega)}\]
for the range of $s$ as in the statement of the theorem. The discrete Poincar\'e inequality, Lemma \ref{l:Poincare2}, immediately implies the asserted error bound.
\end{proof}

\begin{proof}[Proof of Theorem \ref{t:mainthm2}]
The proof is the same as that of Theorem \ref{t:mainthm}. The only differences are that we work with the inner product
$(\cdot,\cdot)_{L^2_{h,*}(\Omega^h\cup\Gamma^h)}$ instead of $(\cdot,\cdot)_{L^2_{h,\sim}(\Omega^h\cup\Gamma^h)}$, use $\hat E^*$ instead of $\hat E$,
and Lemma \ref{l:Poincare1} instead of Lemma \ref{l:Poincare2}.
\end{proof}

\begin{remark}
By Section \ref{s:discretetrace} we know that there are extensions of the boundary values of $\tilde u$ with controlled $\|\cdot\|_\sim$-norm and $\|\cdot\|_*$-norm respectively. In fact, the optimal such extension is
in both cases the biharmonic extension of the boundary values, i.e., the unique function $V$ with the given boundary values that satisfies $\Delta_h^2V=0$ in $\Omega$. Indeed, if $\psi$ is a function such that $\psi = 0, \, D^h_{0,\nu} \psi = 0 \text{ on }\Gamma$, then
\[\| \nabla_h^2 (V+\psi)\|^2_{L^2_{h,\sim}(\Omega^h\cup\Gamma^h)}=\| \nabla_h^2 V\|^2_{L^2_{h,\sim}(\Omega^h\cup\Gamma^h)}+\| \nabla_h^2\psi\|^2_{L^2_{h,\sim}(\Omega^h\cup\Gamma^h)}+2(\nabla_h^2V,\nabla_h^2\psi)_{L^2_{h,\sim}(\Omega^h\cup\Gamma^h)}\]
and $(\nabla_h^2V,\nabla_h^2\psi)_{L^2_{h,\sim}(\Omega^h\cup\Gamma^h)}= 0$, and similarly for $\|\cdot\|_{L^2_{h,*}(\Omega^h\cup\Gamma^h)}$.
This means that we could assume $\hat E$ to be discretely biharmonic, and this would simplify the proof of Theorem \ref{t:mainthm} slightly.
However, for more general fourth-order elliptic operators one cannot use this fact, so we chose to avoid it here.
\end{remark}
\begin{remark}
Using function space interpolation as in Lemma \ref{l:interpolation} it is possible to deduce the intermediate cases of Theorem \ref{t:mainthm} and \ref{t:mainthm2} from the
borderline cases $s=4$ (or $s=3$) and $s=\frac52+\varepsilon$. Our method of proof for $s=4$ (or $s=3$) however directly yields the desired bounds for all
relevant $s$, without the need to resort to function space interpolation.
\end{remark}

\appendix
\section{Density results}\label{a:density}
This section is concerned with the various definitions of the space $H^s_{(k)}$ in the Introduction. Let us recall what we want to prove.
\begin{lemma}\label{l:density}
Let Let $\Xi=I_1\times\cdots\times I_n$, where $I_j\subset\R$ are (possibly unbounded) open intervals, $s\in\R$, $s\ge0$, and $k\in\N_0$
such that $k+\frac12<s$. Then, the following spaces are equal:
\begin{itemize}
	\item[i)] $H^s_{(k)}(\Xi)$, the space of all $u\in H^s(\Xi)$ such that the traces of $\partial_\nu^iu$ for $0\le i\le k$ vanish on $\partial\Xi$;
	\item[ii)] $\overline{\left\{u\in C^\infty(\overline\Xi)\colon\partial_\nu^i u=0\text{ on }\partial\Xi\ \forall i\le k\right\}}^{\|\cdot\|_{H^s(\Xi)}}$, the closure in the $H^s(\Xi)$-norm of the set of all functions in $C^\infty(\overline\Xi)$ whose derivatives up to order $k$ vanish on $\partial\Xi$;
	\item[iii)] $H^s(\Xi)\cap H^{\min(k+1,s)}_0(\Xi)$.
\end{itemize}
\end{lemma}
\begin{remark}
This result actually holds in far more generality (with basically the same proof): on the one hand one can replace the condition
$\partial_\nu^iu=0$ for $0\le i\le k$ by the more general condition $\partial_\nu^iu=0$ for $i\in K$, where $K\subset\N$, as long as
$s-\frac12\notin K$. On the other hand one can take $\Xi$ to be any domain with Lipschitz boundary. The only additional difficulty then is to define $\partial_\nu^iu$ in view of the fact that $\nu$ is in general only a measurable function. However if one defines $\partial_\nu^iu$ as the appropriate linear combination of the traces of $\partial^\alpha u$ for $|\alpha|=i$ (cf. \cite[p.\,156]{Mitrea2013}) the results still hold.
\end{remark}
\begin{proof}[Proof of Lemma \ref{l:density}]
As was already remarked in Section \ref{s:Notation}, for the Lipschitz domain $\Xi$, every function in $C_c^\infty(\overline\Xi)$ is the restriction of a function in $C_c^\infty(\R^n)$ to $\Xi$. In particular, we have that $C_c^\infty(\overline\Xi)\subset H^s(\Xi)$.

We will prove the inclusions
\begin{align}
	\overline{\left\{u\in C_c^\infty(\overline\Xi)\colon\partial_\nu^i u=0\text{ on }\partial\Xi\ \forall i\le k\right\}}^{\|\cdot\|_{H^s(\Xi)}}&\subset H^s(\Xi)\cap H^{\min(k+1,s)}_0(\Xi),\label{e:density1}\\
	H^s(\Xi)\cap H^{\min(k+1,s)}_0(\Xi)&\subset H^s_{(k)}(\Xi), \label{e:density2}\\
	H^s_{(k)}(\Xi)&\subset\overline{\left\{u\in C_c^\infty(\overline\Xi)\colon\partial_\nu^i u=0\text{ on }\partial\Xi\ \forall i\le k\right\}}^{\|\cdot\|_{H^s(\Xi)}}.\label{e:density3}
\end{align}
The inclusion \eqref{e:density2} follows immediately from the definitions and standard trace theorems.

Next observe that trivially
\[\overline{\left\{u\in C_c^\infty(\overline\Xi)\colon\partial_\nu^i u=0\text{ on }\partial\Xi\ \forall i\le k\right\}}^{\|\cdot\|_{H^s(\Xi)}}\subset H^s(\Xi),\]
so in order to prove \eqref{e:density1} we only need to verify that
\[\overline{\left\{u\in C_c^\infty(\overline\Xi)\colon\partial_\nu^i u=0\text{ on }\partial\Xi\ \forall i\le k\right\}}^{\|\cdot\|_{H^s(\Xi)}}\subset H^{\min(k+1,s)}_0(\Xi).\]
To see this, it suffices to prove that
\[\overline{\left\{u\in C_c^\infty(\overline\Xi)\colon\partial_\nu^i u=0\text{ on }\partial\Xi\ \forall i\le k\right\}}^{\|\cdot\|_{{\mathbf H^{\min(k+1,s)}}(\Xi)}}\subset H^{\min(k+1,s)}_0(\Xi).\]
This follows from general theory (e.g. \cite[Theorem 3.18]{Mitrea2013}), but it is also easy to verify by direct calculations: we need to check that we
can approximate any $v\in\left\{u\in C_c^\infty(\overline \Xi)\colon\partial_\nu^i u=0\text{ on }\partial\Xi\ \forall i\le k\right\}$ with $C_c^\infty(\Xi)$-functions in the $H^{\min(k+1,s)}$-norm.

The proof of this assertion proceeds as follows.
The assumptions on $v$ imply that the extension $\bar v$ of $v$ by 0 to $\R^n$ is in $C^k(\R^n)$. In addition, $\bar v \in H^{k+1}(\R^n)$. To verify this one can use that all derivatives  of $v$ of order $k$
are continuous, have zero trace, and are in $H^1(\Xi)$. Hence, their extensions by zero belong to $H^1(\R^n)$. This is well known for general Lipschitz domains (and is easily seen by a partition of unity argument and transformation to the half-space situation by composition with a bi-Lipschitz map). Now dilation is continuous in $H^{k+1}(\R^n)$, and hence $v$ can be approximated by $H^{k+1}_0(\Xi)$ functions in the $H^{k+1}$ norm. Thus, $v \in H^{k+1}_0(\Xi)$. Consequently, $v$ can be approximated in the $\|\cdot \|_{H^{k+1}}$ norm, and in particular in the possibly weaker norm $\| \cdot \|_{H^{\min(s, k+1)}}$ by
$C_c^\infty(\overline\Xi)$ functions.

It remains to prove \eqref{e:density3}.
We first observe that
\[
\left\{u\in C_c^\infty(\overline\Xi)\colon\partial_\nu^i u=0\text{ on }\partial\Xi\ \forall i\le k\right\}=C_c^\infty(\overline\Xi)\cap H^s_{(k)}(\Xi).\]
Taking this into account, we need to verify that $C_c^\infty(\overline\Xi)\cap H^s_{(k)}(\Xi)$ is dense in $H^s_{(k)}(\Xi)$. It is easy to see that $C_c^\infty(\overline\Xi)\cap H^s_{(k)}(\Xi)$ is dense in $C^\infty(\overline\Xi)\cap H^s_{(k)}(\Xi)$, so it remains to prove that the latter space is dense in $H^s_{(k)}(\Xi)$. To see this we apply the criterion of Burenkov \cite[Theorem 2 on p.\,49]{Burenkov1998}. The first three assumptions of that theorem are obviously satisfied, and for the fourth we need to check that every $u\in H^s_{(k)}(\Xi)$ of compact support is continuous under translations, which is once again clear.
\end{proof}

\section{Remarks on Interpolation}
In this section we shall collect and discuss various results on interpolation spaces that were used in our work. As we only consider Hilbert spaces, we do not need the theory
of interpolation spaces in its full generality and can make some simplifications.

We consider two separable Hilbert spaces $X$ and $Y$ such that
$X\subset Y$ is dense and the injection is continuous. Then, given $\theta\in[0,1]$, we can consider the
associated interpolation spaces $[X,Y]_\theta:=D(\Lambda^{1-\theta})$ equipped with the graph norm, where $\Lambda$ is a self-adjoint positive operator on $Y$ with domain $X$ (see \cite[Section 2]{Lions1972} for details, starting with the nontrivial fact that such a $\Lambda$ always exists).
Because we are considering Hilbert spaces, this definition yields up to equivalence of norms
the same space as the complex interpolation space $[X,Y]_{[\theta]}$ or the real interpolation space $[X,Y]_{\theta,2}$ (see \cite[Section 14.2 and Section 15]{Lions1972} for proofs). Thus we will be able to freely use results for either of these interpolation techniques from the literature.

Our first task is to study whether the spaces $H^s_{(k)}(\Xi)$ form an interpolation scale, where $\Xi\subset\R^n$ is open and connected.
If $\Xi$ has a smooth boundary this was shown in \cite{Grisvard1967} with an alternative proof in \cite{Lofstrom1992}. However we are interested
in the cases $\Xi=(0,1)^n$ or $\Xi=(0,\infty)^n$, and the two aforementioned proofs do not easily extend to that case. On the other hand, if $\Xi$ has
Lipschitz boundary then there are results concerning the interpolation scales $H^s(\Xi)$ and $H^s_0(\Xi)$ (see e.g. \cite{Bramble1995}), but not for our mixed case.

Fortunately, in our case it is possible to use the fact that our domain is a cartesian product in combination with results from \cite{Lions1972} to give a proof of the desired result by induction on the dimension.

We begin by stating a one-dimensional but vector-valued result that we will need in the proof of the following lemmas.
In addition to the notation from the Introduction, we define $H^s_\#(I)$, where $I\subset\R$ is an open interval such that $0\in I$, as the closure of the linear space of functions $u$ contained in $C^\infty(I)\cap H^s(I)$ with $u'(0)=0$ in the $\|\cdot\|_{H^s}$-norm.

\begin{lemma}\label{l:interpolation1D}
Let $E$ be a separable Hilbert space, $I\subset\R$ a (possibly unbounded) open interval, and $k\in\N$. Let $s\ge t\ge1$, and let $\theta\in(0,1)$.
If $s-\frac12\notin\{0,1,\ldots,k\}$ and $(1-\theta)s-\frac12\notin\{0,1,\ldots, k\}$, then we have that
\begin{equation}\label{eq:interpolation1D_1}
\left[H^s_{(k)}(I,E),L^2(I,E)\right]_\theta=H^{(1-\theta)s}_{(k)}(I,E).
\end{equation}
Furthermore, if $0\in I$, $s\neq\frac32$ and $(1-\theta)s\neq\frac32$, then
\begin{equation}\label{eq:interpolation1D_2}
\left[H^s_\#(I,E),L^2(I,E)\right]_\theta=H^{(1-\theta)s} _\#(I,E).
\end{equation}
\end{lemma}

\begin{proof}

If $E=\R$ then \eqref{eq:interpolation1D_1} is a special case of \cite[Théorème 8.1]{Grisvard1967}.
The Hilbert-space-valued case follows from  a simple general tensorization argument, see Lemma~\ref{le:tensor_hilbert} below.

For \eqref{eq:interpolation1D_2}  it suffices again to consider the case $E= \R$.
The inclusion ``$\subset$'' is straightforward. For the  converse inclusion we adapt the strategy from \cite{Grisvard1967}. Our goal is to construct for any given $f\in H^{(1-\theta)s}_{(k)}(I)$ some $u\in L^2(\R^+,H^s_{(k)}(I)\cap H^{\frac{1}{2\theta}}(\R^+,L^2(I))$ with $u(\cdot,0)=f$ (cf. \cite[Definition 2.2]{Grisvard1967}).

If $s < \frac32$ then $H^s_\#(I) = H^s(I)$ and the assertion follows by standard results. Thus we may assume $s > \frac32$.\\
We first assume that $s-\frac12\notin\N$ and $(1-\theta) s - \frac12 \notin \N$.
We first define the extension $u$ of $f$
on $I \cap \R^+$ and $I \cap R^-$ separately. Let $\eta \in C^\infty([0,1))$ with $\eta = 1$ on $[0, \frac12]$ and set, for $x \in I \cap \R^+$,
\begin{align*}
f_0^\pm(x) &= f(x), \\
f_k &= 0 \quad \text{for }1 \le k < \frac{1}{2\theta} -
\frac12\\
g_1^\pm(0, y) &= 0 ,\\
g_j^\pm &=(\pm 1)^{j} \frac{\partial^j f}{\partial \nu^j}(0) \eta(y) \quad \text{for }j < (1-\theta)s - \frac12\text{ and }j \ne 1,\\
 g^\pm_j(0,y) &= 0 \quad \quad \text{for } (1-\theta) s - \frac12 \le  j < s - \frac12.
\end{align*}
Then, the compatibility conditions in  \cite[Théorème 7.2]{Grisvard1967} are satisfied and thus there
exist
\[u^\pm \in L^2(\R^+, H^s(I \cap \R^\pm)) \cap H^{\frac1{2\theta}}(\R^+,L^2(I \cap \R^\pm))\]
such that
\begin{align*}
u^\pm(0,x) &= f(x),\\
\frac{\partial^j}{\partial \nu^j}   u^\pm(y,0) &= g_j^\pm(y).
\end{align*}
Set $u(y,x) = u^\pm(y,x)$ for $\pm x > 0$. Then, in particular $u'(y, 0) = 0$.
The condition $f \in H^{(1-\theta)s}$ implies that $(-1)^j g_j^- = g_j^+$ for $j < (1-\theta) s - \frac12$
and hence this condition holds for all $j$.
Because $s \notin \N_0 + \frac12$ it follows that $u(\cdot, y) \in H^s(I)$ for all $y > 0$ and
in fact $u(\cdot, y) \in H^s_\#(I)$. Thus,
\[u \in L^2(\R^+, H^s_\#(I)) \cap H^{\frac1{2\theta}}(\R^+,L^2(I ))\]
and $u(0,x) = f(x)$. By \cite[Definition 2.2]{Grisvard1967} we deduce that $f\in\left[H^s_\#(I),L^2(I)\right]_\theta$, and this concludes the proof of \eqref{eq:interpolation1D_2}.

It remains to remove the assumptions $(1-\theta) s - \frac12 \notin\N$ and $s-\frac12\notin\N$. This can easily be handled by using \cite[Theorem 13.3]{Lions1972} and reiteration. For the convenience of the reader we give the details.

Consider the case $s-\frac12\in\N$, but $(1-\theta) s - \frac12 \notin\N$. Let $s_*>s>\frac32$ be such that $s_*-\frac12\notin\N$ and let $\theta_*$ be such that $s=(1-\theta_*) s_*$. By the reiteration theorem \cite[Theorem 6.1]{Lions1972} we have
\[\left[H^s_\#(I),L^2(I)\right]_\theta=\left[\left[H^{s_*}_\#(I),L^2(I)\right]_{\theta_*},L^2(I)\right]_\theta=\left[H^{s_*}_\#(I),L^2(I)\right]_{\theta+\theta_*-\theta\theta_*}\]
and the right-hand side equals $H^{(1-\theta)s}_\#(I)$ by what we have already shown (note that $(1-(\theta+\theta_*-\theta\theta_*))s_*=(1-\theta)s$).

Next consider the remaining case that $(1-\theta) s - \frac12 \in\N$. Choose $\theta_-<\theta<\theta_+$ close enough to $\theta$ such that $(1-\theta_\pm) s - \frac12 \notin\N$ and $\frac32\notin[(1-\theta_+)s,(1-\theta_-)s]$. Let $\tilde\theta$ be such that $\theta=(1-\tilde\theta)\theta_-+\tilde\theta\theta_+$. Again by reiteration and the previous results we have
\[\left[H^s_\#(I),L^2(I)\right]_\theta=\left[\left[H^s_\#(I),L^2(I)\right]_{\theta_-},\left[H^s_\#(I),L^2(I)\right]_{\theta_+}\right]_{\tilde\theta}=\left[H^{(1-\theta_-)s}_\#(I),H^{(1-\theta_+)s}_\#(I)\right]_{\tilde\theta}\]
and it suffices to show that the right-hand side equals $H^{(1-\tilde\theta)(1-\theta_-)s+\tilde\theta(1-\theta_+)s}_\#(I)=H^{(1-\theta)s}_\#(I)$.
To that end, observe that \[H^t_\#(I)=\begin{cases}\{f\in H^t(I)\colon f'(0)=0\}&\mbox{for $t>\frac32$},\\H^t(I)&\mbox{for $t<\frac32$},\end{cases}\] is a closed subspace of $H^t(I)$ of finite codimension for any $t\ne\frac32$. Now \cite[Theorem 13.3]{Lions1972} implies that for $t<t'<\frac32$ or $\frac32<t<t'$ and $\hat\theta\in[0,1]$ we have
\[\left[H^{t'}_\#(I),H^t_\#(I)\right]_{\hat\theta}=H^{(1-\hat\theta)t'+\hat\theta t}_\#(I).
\]
In particular,
\[\left[H^{(1-\theta_-)s}_\#(I),H^{(1-\theta_+)s}_\#(I)\right]_{\tilde\theta}=H^{(1-\tilde\theta)(1-\theta_-)s+\tilde\theta(1-\theta_+)s}_\#(I).
\]
This completes the proof.

\end{proof}

\begin{lemma} \label{le:tensor_hilbert}
Let $X \subset Y$ be Hilbert spaces of real-valued functions, as above. Let $E$ be a separable Hilbert space
and denote by $X \otimes E$ and $Y \otimes E$ the corresponding spaces of $E$-valued functions.
Then, for all $\theta \in (0,1)$,
\[
[X \otimes E, Y \otimes E]_\theta = [X,Y]_\theta \otimes E,
\]
with equivalent norms.
Here for a Hilbert space $Z$ of real-valued functions the  scalar product on $Z \otimes E$ is defined as follows.
If $(e_m)_{m=1}^\infty$ is an orthonormal basis of $E$ and $f = \sum_{m=1}^\infty f_m e_m$, $g = \sum_{m=1}^\infty g_m e_m$, then
\[ (f,g)_{X \otimes E} = \sum_{m=1}^\infty (f_m, g_m)_X.\]
\end{lemma}

\begin{proof} To show the inclusion ``$\supset$'' let $a \in [X,Y]_\theta \otimes E$ and $\delta > 0$.
Then $a_m \in [X,Y]_\theta$ and
by \cite[Definition 2.2]{Grisvard1967})  there exist
 $u_m \in L^2(\R^+, X) \cap H^{\frac1{2 \theta}}(\R^+,Y)$ with $a_m = u_m(0)$ and
\[ \|u_m\|_{L^2(\R^+, X)} + \|u_m\|_{H^{\frac1{2 \theta}}(\R^+,Y)} \le
(1+ \delta)  \| a_m\|_{[X, Y]_\theta},\qquad m=1,2,\ldots.
\]
Taking the square and summing over $m$ we see that
\[
S := \sum_{m=1}^\infty \|u_m\|_{L^2(\R^+, X)}^2 + \|u_m\|_{H^{\frac1{2 \theta}}(\R^+,Y)}^2 \le (1+ \delta)^2 \| a \|^2_{[X, Y]_\theta \otimes E}.\]
Set $u := \sum_{m=1}^\infty u_m$. Since $S < \infty$ we see that $u \in
L^2(\R^+, X \otimes E) \cap H^{\frac1{2 \theta}}(\R^+,Y \otimes E)$.
Thus $a = u(0) \in [X \otimes E, Y \otimes E]_\theta$
and
\[ \|a\|_{[X \otimes E, Y \otimes E]_\theta}^2 \le 2 S \le 2 (1+ \delta)^2 \|a\|_{[X,Y]_\theta \otimes E}^2.\]

The proof of the converse inclusion, ``$\subset$'',  is similar; nevertheless we include it here for the sake of completeness. %

To that end, let $a \in [X \otimes E, Y \otimes E]_\theta$.
Then there is a $u \in L^2(\R^+, X \otimes E) \cap H^{\frac1{2 \theta}}(\R^+,Y \otimes E)$ with
$u(0) = a$ and
\[
\| u \|_{L^2(\R^+, X \otimes E)}
+ \|u \|_{H^{\frac1{2 \theta}}(\R^+,Y \otimes E)} \le (1 + \delta) \|a\|_{ [X \otimes E, Y \otimes E]_\theta}.
\]
In particular $u_m = (u,e_m)_E$ satisfies
$u_m \in  L^2(\R^+, X) \cap H^{\frac1{2 \theta}}(\R^+,Y)$, $m=1,2,\dots$.
Thus $a_m = u_m(0) \in  [X , Y ]_\theta$, and we have
\begin{align*}
 \|a\|_{ [X, Y]_\theta \otimes E}^2 & =
 \sum_{m=1}^\infty  \|a_m\|^2_{[X,Y]_\theta}\\
 &\le  2  \sum_{m=1}^\infty \|u_m\|_{L^2(\R^+, X)}^2 + \|u_m\|_{H^{\frac1{2 \theta}}(R^+,Y)}^2 \\
& =  2 \|u\|_{L^2(\R^+, X)}^2 + 2  \|u\|_{H^{\frac1{2 \theta}}(R^+,Y)}^2 \\
&\le 2 (1+\delta)^2  \|a\|_{ [X \otimes E, Y \otimes E]_\theta}^2.
\end{align*}
That completes the proof of the lemma.
\end{proof}

Now we can establish the desired interpolation results in higher dimensions. For the following lemma, we are interested in the cases $\Xi=\Omega$ and $k=1$ or $\Xi=(0,\infty)^n$ and $k=0$.

\begin{lemma}\label{l:interpolation}
Let $\Xi=I_1\times\cdots\times I_n$, where $I_j\subset\R$ are (possibly unbounded) open intervals. Let $K=\{k_1,\ldots,k_m\}\subset\N_0$. Let $s\ge t\ge0$, and let $\theta\in[0,1]$. If none of $s-\frac12$, $t-\frac12$ and $(1-\theta)s+\theta t-\frac12$ are in $\{0,1,\ldots k\}$, then
\begin{equation}\label{eq:interpolation1}
\left[H^s_{(k)}(\Xi),H^t_{(k)}(\Xi)\right]_\theta=H^{(1-\theta)s+\theta t}_{(k)}(\Xi),
\end{equation}
and, in particular, if $s\notin\{0,1,\ldots k\}$, $(1-\theta)s-\frac12\notin\{0,1,\ldots k\}$, then
\begin{equation}\label{eq:interpolation2}
\left[H^s_{(k)}(\Xi),L^2(\Xi)\right]_\theta=H^{(1-\theta)s}_{(k)}(\Xi).
\end{equation}
\end{lemma}
\begin{proof}
The identity \eqref{eq:interpolation1} immediately follows from \eqref{eq:interpolation2} and reiteration, so it suffices to establish \eqref{eq:interpolation2}.

We proceed by induction on $n$. The case $n=1$ was established in Lemma \ref{l:interpolation1D}. Now assume that the theorem holds for $n-1$ dimensions. The following argument is similar to the one in Section 2.1 in \cite{Lions1972b}.

Let $\Xi'=I_1\times\cdots\times I_{n-1}$, and write $x=(x',x_n)$. If we interpret a function $\Xi\to\R$ as a function $I_n\to(\Xi'\to\R)$, we claim that
\begin{equation}\label{eq:interpolation3}
L^2(\Xi)=L^2(I_n,L^2(\Xi'))
\end{equation}
and
\begin{equation}\label{eq:interpolation4}
H^s_{(k)}(\Xi)=L^2(I_n,H^s_{(k)}(\Xi'))\cap H^s_{(k)}(I_n,L^2(\Xi')).
\end{equation}
Indeed, \eqref{eq:interpolation3} is obvious. For \eqref{eq:interpolation4} one can argue as follows. It is well-known (and can be proved using the Fourier transform, for example,) that
\begin{equation}\label{eq:interpolation5}
H^s(\R^n)=L^2(\R,H^s(\R^{n-1}))\cap H^s(\R,L^2(\R^{n-1})).
\end{equation}
The sets $\Xi'$ and $I_n$ have Lipschitz boundary, and so there exists an extension operator $E$, mapping $H^t(\Xi)$ continuously to $H^t(\R^n)$ for $t\in\{0,s\}$, that also maps $H^t(\Xi'\times\{x_n\})$ continuously to $H^t(\R^{n-1}\times\{x_n\})$ for any $x_n\in I_n$. One can construct such an $E$ by first applying an appropriate extension operator on each slice $\Xi'\times\{x_n\}\subset\R^{n-1}\times\{x_n\}$ and then extending in the direction $e_n$. Using this extension operator, one can easily check that \eqref{eq:interpolation5} implies also that
\[H^s(\Xi)=L^2(I_n,H^s(\Xi'))\cap H^s(I_n,L^2(\Xi')).\]
From this we want to deduce \eqref{eq:interpolation4} by considering the faces of $\Xi$ separately. We begin with ``$\subset$'' in \eqref{eq:interpolation4}. Let $u\in H^s_{(k)}(\Xi)$, and take $j\le n-1$. Let $\Xi_{j,\pm}'\times I_n$ be the two faces of $\Xi$ orthogonal to $e_j$. By assumption the trace of $\partial_j^iu$ for $i\le k$ vanishes on $\Xi_{j,\pm}'\times I_n$ as an element of $H^{s-j-1/2}(\Xi_{j,\pm}'\times I_n)$ and thus also as an element of $L^2(I_n,H^{s-j-1/2}(\Xi_{j,\pm}')$. In particular, for almost every $x_n$ the trace of $\partial_j^iu(\cdot,x_n)$ vanishes on $\Xi_{j,\pm}'\times \{x_n\}$. This holds for all $j\le n-1$ and all $i\le k$, and so $u\in L^2(I_n,H^s_{(k)}(\Xi'))$. We can argue similarly for the case $j=n$ to deduce that $u\in H^s_{(k)}(I_n,L^2(\Xi'))$ and have thus shown ``$\subset$'' in \eqref{eq:interpolation4}. The argument for ``$\supset$'' is analogous. Thus we have established \eqref{eq:interpolation4}.

Now $L^2(I_n,H^s_{(k)}(\Xi'))$ is the domain of an unbounded positive operator $\Lambda_1$ on $L^2(I_n,L^2(\Xi'))$, and $\Lambda_1$ is an operator in $x'$, independent of $x_n$. Similarly, $H^s_{(k)}(I_n,L^2(\Xi'))$ is the domain of an unbounded positive operator $\Lambda_2$ on $L^2(I_n,L^2(\Xi'))$, and $\Lambda_2$ is an operator in $x'$, independent of $x_n$. In particular, $\Lambda_1$ and $\Lambda_2$ commute. Thus we can apply the criterion for the interpolation space of an intersection \cite[Theorem 13.1]{Lions1972} and obtain that
\begin{align*}
&\left[H^s_{(k)}(\Xi),L^2(\Xi)\right]_\theta\\
&\quad=\left[L^2(I_n,H^s_{(k)}(\Xi'))\cap H^s_{(k)}(I_n,L^2(\Xi')),L^2(I_n,L^2(\Xi'))\right]_\theta\\
&\quad=\left[L^2(I_n,(H^s_{(k)}(\Xi')),L^2(I_n,L^2(\Xi'))\right]_\theta\cap\left[H^s_{(k)}(I_n,L^2(\Xi')),L^2(I_n,L^2(\Xi'))\right]_\theta.
\end{align*}
Now, according to \cite[Remark 14.4]{Lions1972} and using the induction hypothesis \eqref{eq:interpolation2} for $n-1$, we have
\begin{align*}\left[L^2(I_n,H^s_{(k)}(\Xi')),L^2(I_n,L^2(\Xi'))\right]_\theta&=L^2\left(I_n,\left[H^s_{(k)}(\Xi'),L^2(\Xi')\right]_\theta\right)\\
&=L^2(I_n,H^{(1-\theta)s}_{(k)}(\Xi')).
\end{align*}
Similarly, using \eqref{eq:interpolation1D_1}, we find
\[
\left[H^s_{(k)}(I_n,L^2(\Xi')),L^2(I_n,L^2(\Xi'))\right]_\theta=H^{(1-\theta)s}_{(k)}(I_n,L^2(\Xi')).
\]
If we combine the last three equalities, we deduce that
\begin{align*}
\left[H^s_{(k)}(\Xi),L^2(\Xi)\right]_\theta&=L^2(I_n,H^{(1-\theta)s}_{(k)}(\Xi'))\cap H^{(1-\theta)s}_{(k)}(I_n,L^2(\Xi'))\\
&=H^{(1-\theta)s}_{(k)}(\Xi).
\end{align*}
That completes the proof of the lemma.
\end{proof}

For the next lemma recall the definition of $G^s$ from the proof of Lemma \ref{l:estboundaryvalues2}.

\begin{lemma}\label{l:interpolation2}
Let $s\ge t\ge0$, and let $\theta\in(0,1)$. Then, if none of $s$, $t$ and $(1-\theta)s+\theta t$ are in $\left\{\frac12,\frac32\right\}$, we have
\begin{equation}\label{eq:interpolation2_1}
\left[G^s((0,\infty)^{n-1}\times\R),G^t((0,\infty)^{n-1}\times\R)\right]_\theta=G^{(1-\theta)s+\theta t}((0,\infty)^{n-1}\times\R),
\end{equation}
and in particular, if $s\notin\left\{\frac12,\frac32\right\}$, $(1-\theta)s\notin\left\{\frac12,\frac32\right\}$, then
\begin{equation}\label{eq:interpolation2_2}
\left[G^s((0,\infty)^{n-1}\times\R),L^2((0,\infty)^{n-1}\times\R)\right]_\theta=G^{(1-\theta)s}((0,\infty)^{n-1}\times\R).
\end{equation}
\end{lemma}
\begin{proof}
As in the previous lemma, \eqref{eq:interpolation2_1} follows from \eqref{eq:interpolation2_2} and reiteration, so we will only prove \eqref{eq:interpolation2_2}. Observe that
\[G^s((0,\infty)^{n-1}\times\R)=L^2(\R,H^s_{(0)}((0,\infty)^{n-1}))\cap H^s_\#(\R,L^2((0,\infty)^{n-1})\]
and
\[L^2((0,\infty)^{n-1}\times\R)=L^2(\R,L^2((0,\infty)^{n-1})).\]

Intersection and interpolation commute by the same argument as in the proof of Lemma \ref{l:interpolation}, and so we have
\begin{align*}
&\left[G^s((0,\infty)^{n-1}\times\R),L^2((0,\infty)^{n-1}\times\R)\right]_\theta\\
&\quad=\left[L^2(\R,H^s_{(0)}((0,\infty)^{n-1}),L^2(\R,L^2((0,\infty)^{n-1}))\right]_\theta\cap \left[H^s_\#(\R,L^2((0,\infty)^{n-1}),L^2(\R,L^2((0,\infty)^{n-1}))\right]_\theta.
\end{align*}
Now, by Lemma \ref{l:interpolation} we have
\[\left[L^2(\R,H^s_{(0)}((0,\infty)^{n-1}),L^2(\R,L^2((0,\infty)^{n-1}))\right]_\theta=L^2(\R,H^{(1-\theta)s}_{(0)}((0,\infty)^{n-1}),\]
and Lemma \ref{l:interpolation1D} implies that
\[\left[H^s_\#(\R,L^2((0,\infty)^{n-1}),L^2(\R,L^2((0,\infty)^{n-1}))\right]_\theta=H^{(1-\theta)s}_\#(\R,L^2((0,\infty)^{n-1}).\]
The last three equalities combined imply \eqref{eq:interpolation2_2}.
\end{proof}

\textbf{Acknowledgements.} FS gratefully acknowledges the kind hospitality of Oxford University during a short visit in April 2018.\\
SM and FS were partially supported by the Excellence Clusters EX 59 and EX 247-1, the Hausdorff Center for Mathematics and the SFB 1060 'The Mathematics of Emergent Effects'. FS was also supported by the German National Academic Foundation.

\bibliographystyle{alpha}
\bibliography{optimal_error_biharmonic_mar29_2019}

\begin{thebibliography}{GLMP83}

\bibitem[BACF13]{Ben-Artzi2013}
M.~Ben-Artzi, J.-P. Croisille, and D.~Fishelov.
\newblock {\em Navier--{S}tokes equations in planar domains}.
\newblock Imperial College Press, London, 2013.

\bibitem[Bra66]{Bramble1966}
J.~H. Bramble.
\newblock A second order finite difference analog of the first biharmonic
  boundary value problem.
\newblock {\em Numer. Math.}, 9:236--249, 1966.

\bibitem[Bra95]{Bramble1995}
J.~H. Bramble.
\newblock Interpolation between {S}obolev spaces in {L}ipschitz domains with an
  application to multigrid theory.
\newblock {\em Math. Comp.}, 64(212):1359--1365, 1995.

\bibitem[Bur98]{Burenkov1998}
V.~I. Burenkov.
\newblock {\em Sobolev spaces on domains}, volume 137 of {\em Teubner-Texte zur
  Mathematik [Teubner Texts in Mathematics]}.
\newblock B. G. Teubner Verlagsgesellschaft mbH, Stuttgart, 1998.

\bibitem[Ehr71]{Ehrlich1971}
L.~W. Ehrlich.
\newblock Solving the biharmonic equation as coupled finite difference
  equations.
\newblock {\em SIAM J. Numer. Anal.}, 8:278--287, 1971.

\bibitem[GLMP83]{Gavrilyuk1983b}
I.~P. Gavrilyuk, R.~D. Lazarov, V.~L. Makarov, and S.~P. Pirnazarov.
\newblock Estimates for the rate of convergence of difference schemes for
  fourth-order equations of elliptic type.
\newblock {\em Zh. Vychisl. Mat. i Mat. Fiz.}, 23(2):355--365, 1983.

\bibitem[GMP83]{Gavrilyuk1983}
I.~P. Gavriljuk, V.~L. Makarov, and S.~P. Pirnazarov.
\newblock {Consistent estimates of the rate of convergence of difference
  solutions to generalized solutions of the first boundary value problem for
  fourth-order equations}.
\newblock {\em Izv. Vyssh. Uchebn. Zaved. Mat.}, (2):15--22, 1983.
\newblock (Russian, transl. {\it Soviet Math. (Iz. VUZ)} 27(2), 13--21 (1983)).

\bibitem[Gri67]{Grisvard1967}
P.~Grisvard.
\newblock Caract\'{e}risation de quelques espaces d'interpolation.
\newblock {\em Arch. Rational Mech. Anal.}, 25:40--63, 1967.

\bibitem[Hac81]{Hackbusch1981}
W.~Hackbusch.
\newblock {On the regularity of difference schemes}.
\newblock {\em Ark. Mat.}, 19(1):71--95, 1981.

\bibitem[I{\u I}S86]{Ivanovich1986}
L.~D. Ivanovich, B.~S. {\u I}ovanovich, and E.~S\"{u}li.
\newblock Convergence of difference schemes for the biharmonic equation.
\newblock {\em Zh. Vychisl. Mat. i Mat. Fiz.}, 26(5):776--779, 799, 1986.

\bibitem[JS14]{Jovanovic2014}
B.~S. Jovanovi{\'c} and E.~S{\"u}li.
\newblock {\em {Analysis of finite difference schemes for linear partial
  differential equations with generalized solutions}}, volume~46 of {\em
  {Springer Series in Computational Mathematics}}.
\newblock Springer, London, 2014.

\bibitem[Laz81]{Lazarov1981}
R.~D. Lazarov.
\newblock On the convergence of difference solutions to generalized solutions
  of a biharmonic equation in a rectangle.
\newblock {\em Differentsial'nye Uravneniya}, 17(7):1295--1303, 1344, 1981.

\bibitem[LM72a]{Lions1972}
J.-L. Lions and E.~Magenes.
\newblock {\em Non-homogeneous boundary value problems and applications. {V}ol.
  {I}}.
\newblock Springer-Verlag, New York-Heidelberg, 1972.
\newblock Translated from the French by P. Kenneth, Die Grundlehren der
  mathematischen Wissenschaften, Band 181.

\bibitem[LM72b]{Lions1972b}
J.-L. Lions and E.~Magenes.
\newblock {\em Non-homogeneous boundary value problems and applications. {V}ol.
  {II}}.
\newblock Springer-Verlag, New York-Heidelberg, 1972.
\newblock Translated from the French by P. Kenneth, Die Grundlehren der
  mathematischen Wissenschaften, Band 182.

\bibitem[L{\"o}f92]{Lofstrom1992}
J.~L{\"o}fstr{\"o}m.
\newblock Interpolation of boundary value problems of {N}eumann type on smooth
  domains.
\newblock {\em J. London Math. Soc. (2)}, 46(3):499--516, 1992.

\bibitem[MM13]{Mitrea2013}
I.~Mitrea and M.~Mitrea.
\newblock {\em Multi-layer potentials and boundary problems for higher-order
  elliptic systems in {L}ipschitz domains}, volume 2063 of {\em Lecture Notes
  in Mathematics}.
\newblock Springer, Heidelberg, 2013.

\bibitem[MS19]{Muller2019}
S.~M\"{u}ller and F.~Schweiger.
\newblock Estimates for the {G}reen's {F}unction of the {D}iscrete
  {B}ilaplacian in {D}imensions 2 and 3.
\newblock {\em Vietnam J. Math.}, 47(1):133--181, 2019.

\bibitem[Sch19]{Schweiger2019}
F.~Schweiger.
\newblock The maximum of the four-dimensional membrane model, 2019.
\newblock {a}r{X}iv:1903.02522.

\bibitem[Smi68]{Smith1968}
J.~Smith.
\newblock The coupled equation approach to the numerical solution of the
  biharmonic equation by finite differences. {I}.
\newblock {\em SIAM J. Numer. Anal.}, 5:323--339, 1968.

\bibitem[Smi70]{Smith1970}
J.~Smith.
\newblock The coupled equation approach to the numerical solution of the
  biharmonic equation by finite differences. {II}.
\newblock {\em SIAM J. Numer. Anal.}, 7:104--111, 1970.

\bibitem[Tee64]{Tee1963}
G.~J. Tee.
\newblock A novel finite-difference approximation to the biharmonic operator.
\newblock {\em Comput. J.}, 6:177--192, 1963/1964.

\bibitem[Whi34]{Whitney1934}
H.~Whitney.
\newblock Functions differentiable on the boundaries of regions.
\newblock {\em Ann. of Math. (2)}, 35(3):482--485, 1934.

\end{thebibliography}

\end{document}